\documentclass[a4paper,11pt]{article}

\usepackage[margin=1in]{geometry}

\usepackage{amsmath}
\usepackage{amssymb}
\usepackage{amsfonts}
\usepackage{amsthm}
\usepackage{mathrsfs} 
\usepackage{mathtools}
\usepackage{graphicx}
\usepackage{color}
\usepackage{bm}
\usepackage{xcolor}
\usepackage{hyperref}
\usepackage{fontenc}
\usepackage[normalem]{ulem}
\usepackage[mathlines]{lineno}
\usepackage{subcaption}

\usepackage{algorithm}
\usepackage{algpseudocode}

\newtheorem{theorem}{Theorem}

\newtheorem{proposition}{Proposition}
\newtheorem{corollary}{Corollary}

\newtheorem{definition}{Definition}
\newtheorem{assumption}{Assumption}
\newtheorem{remark}{Remark}

\newtheorem{example}{Example}

\newcommand{\mk}[1]{{\color{black}  #1}}

\newcommand{\ks}[1]{{\color{black} #1}}
\definecolor{chatdarkgreen}{RGB}{0,100,0}
\newcommand{\chat}[1]{{\color{black}#1}}
\definecolor{claudepurple}{RGB}{128,0,128}
\newcommand{\claude}[1]{{\color{black}#1}}

\newcommand{\M}{\mathcal M}
\newcommand{\X}{\mathcal X}
\newcommand{\Xzero}{\mathcal X_0}

\newcommand{\Rp}{[0, \infty)}
\newcommand{\Sinf}{\mathcal S_\infty}

\newcommand{\R}{\mathbb{R}} 
\newcommand{\N}{\mathbb{N}} 

\usepackage[
style=ieee,
sorting=nyt
]{biblatex}
\newcommand\Set[2]{\left\{\,#1\mbox{ }\middle|\mbox{ }#2\,\right\}}
\newcommand{\defeq}{:=}

\newcommand{\mass}{\text{mass }}

\title{\bf Bounding extreme events in dynamical systems over an infinite time horizon}

\begin{document}

\author{Karol\'{\i}na Sehnalov\'a$^1$ \and   Milan Korda$^{1,2}$ }

\footnotetext[1]{Faculty of Electrical Engineering, Czech Technical University in Prague, Technick\'a 2, CZ-16627 Prague, Czechia.}
\footnotetext[2]{CNRS; LAAS; Universit\'e de Toulouse, 7 avenue du colonel Roche, F-31400 Toulouse, France. }




\date{Draft of \today}

\maketitle

\begin{abstract}
This work addresses the computation of extreme values of observables along trajectories of  dynamical systems over an \ks{infinite time horizon}. The setting is intrinsically global: trajectories may exhibit large transient excursions, approach their largest values only asymptotically as time tends to infinity, or escape to infinity in finite time. The proposed framework treats these \mk{ scenarios} in a unified way, so that the resulting extreme values reflect both the transient and the asymptotic behavior of the dynamics. The main difficulty is that neither the time interval nor the state space is assumed to be compact. We address this by combining a radial compactification of the state space with a compactified clock variable and discounted occupation measures. This yields a  description of trajectory ensembles on compact sets via occupation measures and the associated Liouville's equation. Extreme values are then characterized through a family of reachability problems for level sets of the observable, which can be solved by a bisection procedure. On the theoretical side, we prove a superposition principle for an infinite-horizon Liouville's equation with a time-distributed terminal measure. This result shows that the measure formulation represents ensembles of stopped admissible trajectories, including trajectories \ks{that never stop}, and establishes equivalence between the trajectory-level and measure-level descriptions.

\end{abstract}

\section{Introduction}
Estimating extreme values attained along trajectories of dynamical systems is a classical problem in dynamical systems, with applications ranging from safety verification and reachability to performance identification and robustness analysis. In this work, we consider the extreme value estimation problem for autonomous polynomial dynamical systems of the form
\begin{equation*}
    \dot x(t) = f(x(t)), \qquad x(0)\in \mathcal X_0 \subset \R^n,\qquad t \in [0, \infty),
\end{equation*}
where $f$ is a polynomial vector field and $\mathcal X_0$ is a compact semialgebraic set. The extreme value of interest is defined as a supremum over admissible trajectories $x$
\[
    \sup_{x(0) \in \mathcal X_0,\; t\ge 0 } \|x(t)\|,
\]
where $\|\cdot\|$ denotes a Euclidean norm. More generally, we are interested in the supremum of a polynomial observable $p(x)$ along trajectories
\[
\sup_{x(0) \in \mathcal X_0,\; t\ge 0 }p(x(t)).
\]

The fundamental difficulty addressed in this paper is that both the time horizon and the state space are a priori unbounded. For polynomial systems, unbounded trajectories and finite-time blow-up phenomena may occur \cite{matsue_blow-up_2018, matsue_blow-up_2025, elias_critical_2006}, in which case the extreme value of the norm is infinite. Even when all trajectories remain bounded, the extreme value may be attained asymptotically as $t\to\infty$ (which is the case e.g. for the Van der Pol oscillator \cite{khalil}); this precludes a direct reduction to finite-horizon optimization. Thus, extreme value estimation on infinite time horizons poses substantial analytical and numerical challenges. In this paper, we develop a method for estimation of the maximal norm of the trajectories of an autonomous system with polynomial dynamics. Once we obtain a bound on the norm, we present a way to estimate the extreme value of a general polynomial observable $p(x)$.

A fruitful approach to estimating quantitative trajectory behaviour is based on the formulation of dynamical constraints via linear equations on measures, commonly referred to as occupation measure relaxations (see for example \cite{lasserre_nonlinear_2008, korda_convex_2014}). Its main advantage is the ability to directly analyze the behaviour of an ensemble of trajectories simultaneously and its solvability by a tractable moment-SOS hierarchy. Using Liouville's (continuity) equation \cite{lewis_relaxation_1980, vinter_1993, diperna_ordinary_1989, bernard_young_2008, buehrle_optimal_2025, unbounded_ctrl}, nonlinear dynamics can be encoded as a linear constraint on measures, allowing extreme value estimation problems to be cast as infinite-dimensional linear programs (LP) in the space of measures.

On finite time horizons and compact state spaces, this framework has already been successfully developed using the moment--SOS hierarchy in \cite{fantuzzi_bounding_2020} in the dual setting of auxiliary functions and in~\cite{miller_peak_2021, miller_peak1_2021} using the measure formulation. On infinite time horizons, extreme values on global attractors have been upper bounded in \cite{goluskin2020bounding}. \ks{However, even in the case of finite peak value attained at finite time, the existing methods usually require a priori knowledge of a bound on the state norm and time interval. Our proposed method can be used to obtain such bounds on the state norm.}

Our approach combines finding upper bounds attained at finite times (i.e., transient behavior) and asymptotically as $t \to \infty$ in one general framework. In our setting, two major obstacles arise. First, the occupation measure typically has infinite mass on infinite horizons, and second, the state space and the time horizon are noncompact, which prevents a direct application of the moment-SOS hierarchy. The main contribution of this paper is a measure-theoretic formulation of extreme value estimation that overcomes both difficulties at the price of arriving at not a single measure LP, but a family of measure LPs.

Our approach relies on three key ideas. First, we employ a Poincar\'e-type compactification \cite{matsue_blow-up_2018, elias_critical_2006, takayasu_numerical_2017} of the state space, which lifts the unbounded state variable $x\in\mathbb{R}^n$ to a compact manifold embedded in $\mathbb{R}^{n+1}$. This is, in fact, the radial compactification of $\R^n$, which enables us to distinguish between different directions towards infinity \cite{takayasu_numerical_2017}.  Second, we introduce a discounted clock variable, which compactifies the time axis and yields a Liouville's equation posed entirely on a compact domain. In order to ensure a finite mass of the occupation measure, we also discount the Liouville's equation directly by adding an additional sink term, which dissipates the mass of the occupation measure. This approach corresponds to using the so-called discounted occupation measures, which are commonly used in stochastic dynamical systems (e.g., \cite{bhatt_occupation_1996}). \ks{They were also used for infinite-horizon optimal control problems with discounted functionals in \cite{kamoutsi_2017}}. Third, instead of maximizing a terminal cost (which is discounted when we are using the discounted Liouville's equation), we formulate extreme value estimation as a sequence of reachability tests for level sets of the state norm. This leads naturally to a bisection procedure that detects whether a given level is reachable by at least one admissible trajectory.

At the infinite-dimensional level, this reachability formulation is expressed as an LP over measures constrained by the discounted Liouville's equation with compact supports. A key theoretical contribution of the paper is a superposition principle adapted to this setting. We prove that measure solutions of the infinite-horizon Liouville's equation with a time-distributed terminal measure can be represented by ensembles of stopped admissible trajectories, including trajectories with infinite stopping time. This result justifies the measure formulation and, in particular, shows that feasibility with a terminal measure of positive mass is equivalent to the existence of a trajectory reaching the prescribed level set. This equivalence provides a rigorous foundation for the proposed bisection-based extreme value recovery method on infinite horizons and is an instance of a \textit{no relaxation gap} result in measure relaxations.

On the numerical side, the resulting measure LPs are defined on compact semialgebraic sets and admit systematic and tractable approximation via the moment--SOS hierarchy. For each candidate level, an infinite-dimensional LP certifies either reachability or non-reachability. This infinite dimensional LP is then numerically solved by a hierarchy of SDPs (semidefinite programs). Once a finite bound on the state norm has been obtained, the same methodology extends naturally to the extreme value estimation of arbitrary polynomial observables instead of focusing solely on the norm.

\subsection{Outline}
The paper is organized as follows. Section~\ref{sec_problem_statement} formally introduces the problem setting. Section~\ref{sec:compactification_reparametrization} presents the state compactification and the associated polynomial dynamics after time reparameterization. Section~\ref{sec_occupation_measures} develops the occupation measure formulation on infinite horizons, including different versions of Liouville's equation used to represent the dynamics of the system. Section~\ref{sec:peak_recovery} introduces the reachability-based extreme value estimation method
and establishes its correctness. Finally, Section~\ref{sec:examples} illustrates the proposed method on numerical examples including the Van der Pol oscillator.

\subsection{Notations}
In this work, measures are assumed to be nonnegative Radon measures with an exception in the appendix, where we use signed Radon measures. A set of nonnegative Radon measures (resp. probability measures) supported on a set $\Omega$ is denoted by $\mathcal M_+(\Omega)$ (resp. $\mathcal P(\Omega)$). The Lebesgue measure in the $x$ variable is denoted $\lambda(dx)$ or $dx$ and the Dirac measure at a point $s$ is denoted by $\delta_s(dx)$ or $\delta_s$. Given a measurable map \(T : \mathcal A \to \mathcal B\) and a measure \(\mu\) on $\mathcal A$, the pushforward measure \(T_{\#}\mu\) is the measure on \(\mathcal B\) defined by
$(T_{\#}\mu)(\mathcal C) = \mu\!\left(T^{-1}(\mathcal C)\right)$
for every measurable set \(\mathcal C \subset \mathcal B\). Whenever a measure is written with differential notation, e.g. \(\mu(dt,dx)\) or $d\mu(t, x)$, we refer to the measure itself. In contrast, we use \(\mu(t)\) to denote the conditional (slice) measure at time \(t\) arising from the disintegration of \(\mu\) with respect to its time marginal. We denote the radial compactification of $\R^n$ by $\overline{\R^n}$. By $[0, \infty]$ we understand $\R_+ \cup \{+\infty\}$, the one point compactification of $\R_+ = [0, \infty)$. By $\|\bullet \|$ we denote the Euclidean norm of a vector and a dot product of two vectors $x, y \in \R^n$ is denoted by $x\cdot y$. Given a set $\mathcal A \subset \R^n$, we denote its closure by $\overline{\mathcal A}$ \ks{and its complement by $\mathcal A^c$}. By $C(\mathcal A)$ and $C^1(\mathcal A)$ we denote the spaces of continuous and continuously differentiable functions with domain $\mathcal A$, respectively; the subscript $c$ denotes the corresponding subspace of compactly supported functions. When it is necessary to specify the codomain $\mathcal B$, we write $C(\mathcal A, \mathcal B)$, $C_c(\mathcal A, \mathcal B)$, $C^{1}(\mathcal A, \mathcal B)$, $C^{1}_c(\mathcal A, \mathcal B)$. \chat{The subscript $b$ denotes the corresponding space of bounded functions.} By $\R[x]$ we understand the vector space of all functions $f(x)$, which are polynomial in $x$. By $\mathrm{deg}(f)$ we denote the degree of the polynomial function $f(x)$.

\section{Problem statement}
\label{sec_problem_statement}

Consider the dynamical system $\dot x(t) = f(x)$ on a time interval $[0, \infty)$, where $x = (x_1, \ldots, x_n)\in \R^n$, $n\in \N$ and the dynamics $f$ is a polynomial. \chat{We set $d:=\max\{1,\mathrm{deg}(f_1),\ldots,\mathrm{deg}(f_n)\}$, with the convention $\mathrm{deg}(0):=0$.} The problem we address in this paper is as follows:
\begin{align}
\label{problem_state_original}
    &\bar{p}^* = \sup_{\tau, x} \|x(\tau)\|\notag \\
    \text{s.t.: } &\dot{x}(t) = f(x(t)),\quad t \in [0, \tau], \\ &x(0) \in \mathcal X_0,\notag
\end{align}
where $\mathcal X_0 \subset \R^n$ is a compact basic semialgebraic set defined as
\begin{equation*}
    \mathcal X_0 = \Set{x \in \R^n}{g_i(x) \geq 0, \quad i = 1,\ldots, N},
\end{equation*}
where $g_i(x)$ are polynomials, $N\in\N$. \chat{We assume throughout that $\mathcal X_0\neq\varnothing$.}

Note in particular that the extreme value may be finite or infinite and the supremum may or may not be attained at a finite time. In particular, we do not rule out a finite-time escape of the trajectories. Observe that the problem \eqref{problem_state_original} is feasible, since the polynomial system $\dot x = f(x)$ is locally Lipschitz and hence \chat{for every $x_0\in\mathcal X_0$ there exists $\tau>0$ such that the ODE has a solution on the closed time interval $[0,\tau]$ with $x(0)=x_0$.}

We reformulate the problem \eqref{problem_state_original} as a linear problem on measures and then compactify, since the state variable is not a priori bounded. Instead of maximizing $\|x(\tau)\|$, we can maximize $\|x(\tau)\|^2$, which is equivalent. To avoid infinite values of the cost function, we can maximize any strictly increasing function $h(\|x\|^2)$ with bounded image. Here we take the function
\begin{align*}
    \|x\|^2 \mapsto \frac{\|x\|^2}{1 + \|x\|^2},
\end{align*}
which is indeed strictly increasing in $\|x\|^2$ and its image is bounded. Notice that any maximizer of $\frac{\|x\|^2}{1 + \|x\|^2}$ with objective value $m^*$ is indeed a maximizer of $\|x\|^2$ and we can recover $\|x\|^2$ from $m^*$ as $\|x\|^2 = \frac{m^*}{1-m^*}$ if $m^* < 1$. Hence, from now on, we focus on the following problem equivalent to \eqref{problem_state_original}
\begin{align}
\label{problem_state_bounded_cost}
     &p^* = \sup_{\tau, x}\frac{\|x(\tau)\|^2}{1 + \|x(\tau)\|^2}\notag \\
    \text{s.t.: } &\dot{x}(t) = f(x(t)), \quad t \in [0, \tau],\\ &x(0) \in \mathcal X_0,\notag
\end{align}
which has the same maximizer (if the maximum is attained), although a different objective value. However, the objective value of the problem \eqref{problem_state_original} can be recovered from the objective value of the problem \eqref{problem_state_bounded_cost} as $\bar p^* = \sqrt{\frac{p^*}{1-p^*}}$ if $p^* < 1$ and $\bar p^* = \infty$ if $p^* = 1$ in view of the discussion above.

\subsection{Illustrative examples}

As mentioned in the Introduction, the proposed approach does not require any prior knowledge of how an extreme event occurs. In particular, it covers three qualitatively different scenarios.

First, a finite extreme value may be attained at a finite time. This setting was previously studied in \cite{miller_peak1_2021} and is also covered by our formulation. However, our method does not require any a priori known finite time horizon for the occurrence of the extreme value. The following polynomial example illustrates this behaviour.

\begin{example}
Consider the system
\[
\begin{aligned}
\dot{x}_1
&=
\frac{x_2}{2}
-2\kappa x_1
\left(
x_1^2+\frac{x_2^2}{4}-1
\right),\\
\dot{x}_2
&=
-2x_1
-\frac{\kappa x_2}{2}
\left(
x_1^2+\frac{x_2^2}{4}-1
\right).
\end{aligned}
\]
The ellipse
\(
x_1^2+x_2^2/4=1
\)
is a stable limit cycle. For example, if \(\kappa=0.1\) and \(x(0)=(-1.4,0)\), the trajectory approaches this limit cycle from its exterior, but its Euclidean norm attains the transient peak
\[
\left\|x(t_{\mathrm{peak}})\right\|
\approx 2.453
\qquad\text{at}\qquad
t_{\mathrm{peak}}\approx 1.492.
\]
This exceeds both the initial norm and the maximal norm \(2\) on the limit cycle. The behaviour is illustrated in Figure~\ref{fig:transient-peak}.
\end{example}

Second, a trajectory may exhibit finite-time blow-up, in which case the corresponding extreme value is unbounded. The following elementary example illustrates this behaviour.

\begin{example}
Consider the scalar system
\[
\dot{x}(t)=x(t)^2,
\qquad
x(0)=1.
\]
Its solution is
\[
x(t)=\frac{1}{1-t},
\qquad
t\in[0,1),
\]
and therefore becomes unbounded as \(t\to 1^{-}\). Thus, the solution experiences finite-time blow-up at \(t=1\). The trajectories for different possible initial conditions are shown in Figure \ref{fig:blowup-trajectory}.
\end{example}

Finally, a finite extreme value need not be attained at any finite time; it may instead be approached asymptotically as \(t\to\infty\). This behaviour is illustrated by the Van der Pol oscillator \cite{khalil}.

\begin{example}
Consider the Van der Pol oscillator,
\begin{align*}
\dot{x}_1(t)
&=x_2(t), \\
\dot{x}_2(t)
&=-x_1(t)
+\bigl(1-x_1(t)^2\bigr)x_2(t).
\end{align*}
The system possesses a stable limit cycle. For example, the trajectory starting from
\(
x(0)=\left(\frac{1}{2},\frac{1}{2}\right)
\)
approaches this limit cycle from its interior. In this case, the supremum of the Euclidean norm of the trajectory is given by the maximal norm attained on the limit cycle. This value is approached asymptotically as \(t\to\infty\), but it is not attained at any finite time. This scenario is shown in Figure \ref{fig:vanderpol-trajectory}.
\end{example}

\begin{figure}[h]
\centering

\begin{subfigure}[c]{0.33\textwidth}
    \centering
    \includegraphics[width=\textwidth]{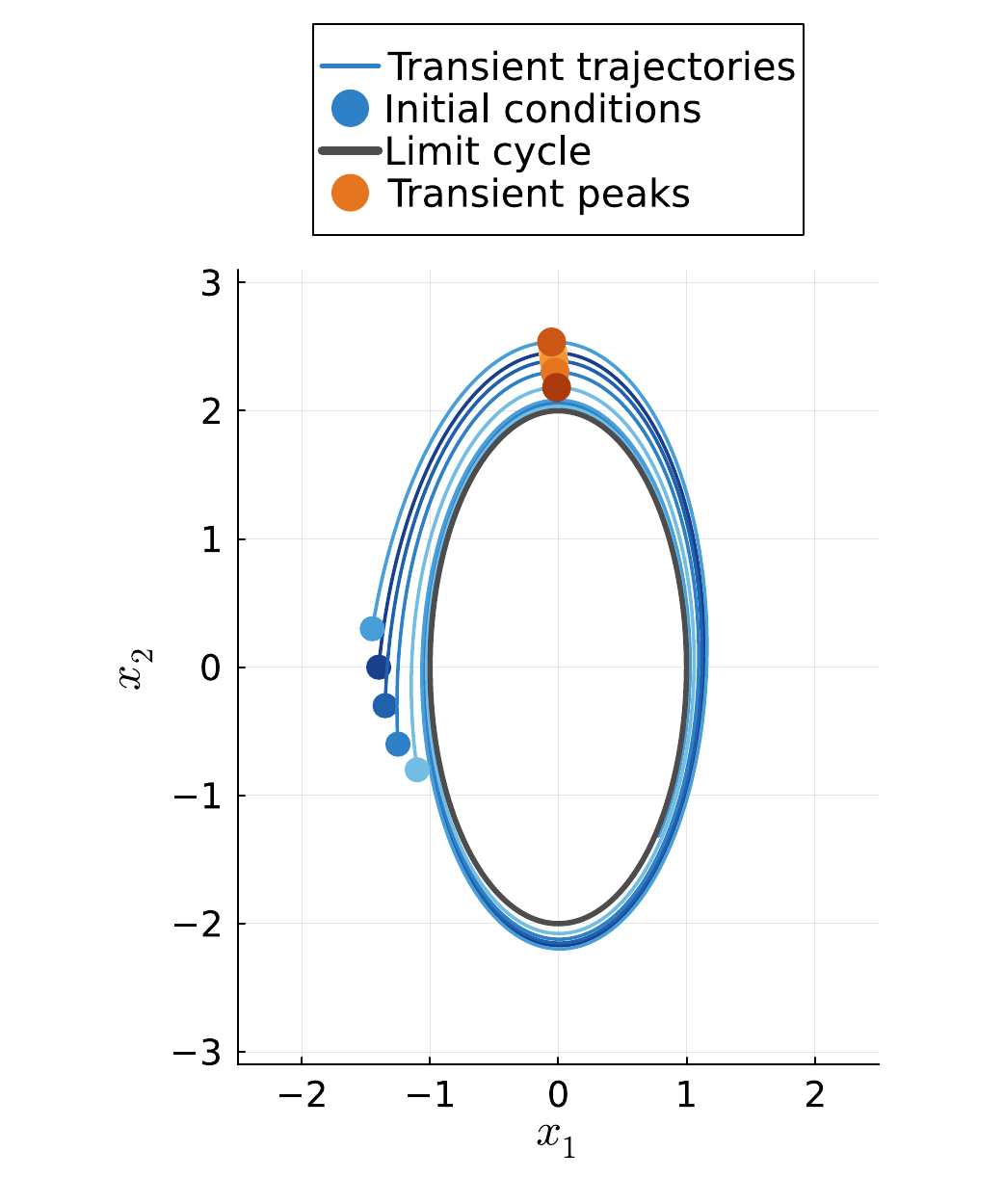}
    \caption{Trajectories of the system attaining finite peaks during the transient period}
    \label{fig:transient-peak}
\end{subfigure}\hfil
\begin{subfigure}[c]{0.32\textwidth}
    \centering
    \includegraphics[width=\textwidth]{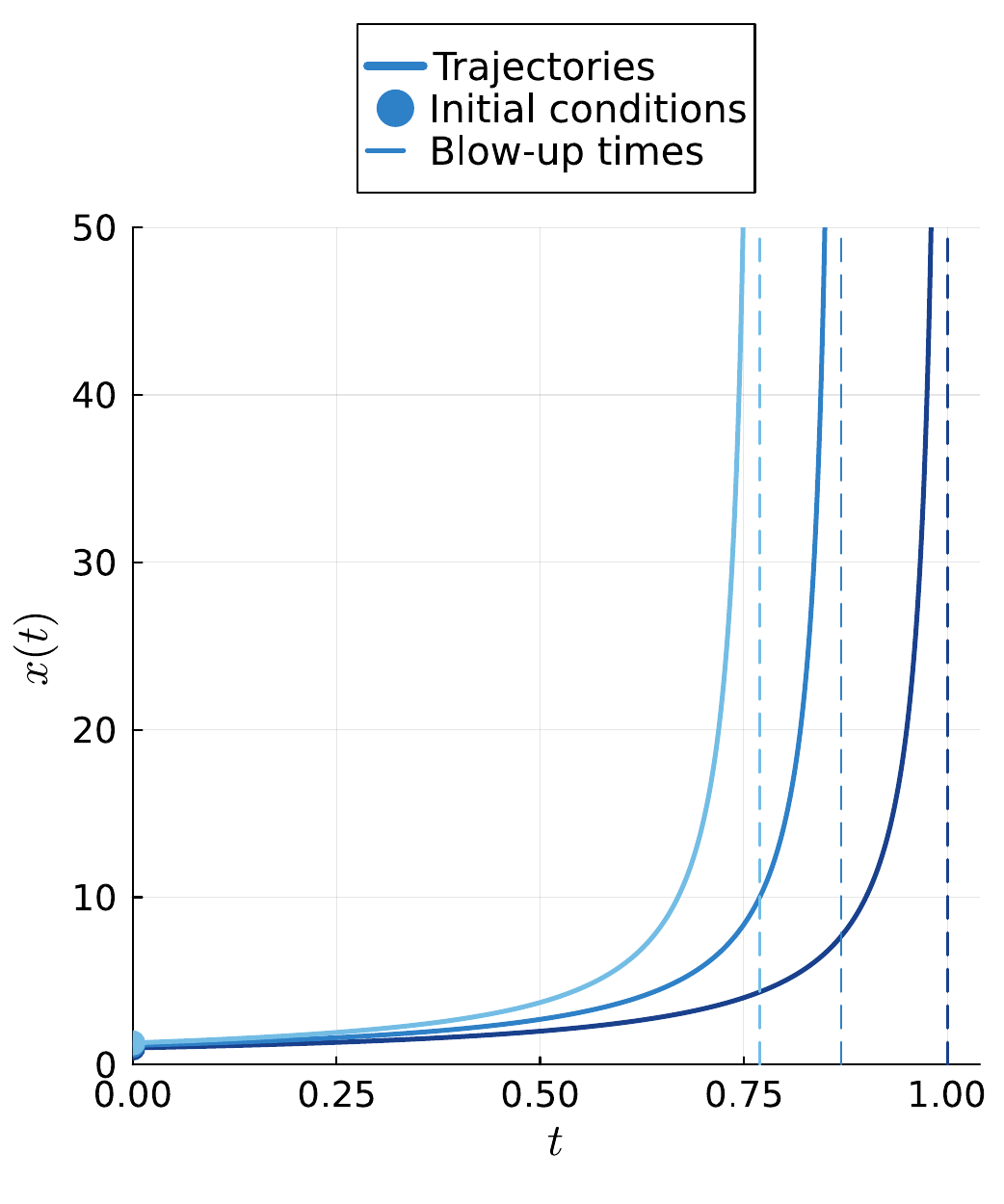}
    \caption{Trajectories of the scalar system exhibiting finite-time blow-up at different times}
    \label{fig:blowup-trajectory}
\end{subfigure}\hfil
\begin{subfigure}[c]{0.33\textwidth}
    \centering
    \includegraphics[width=\textwidth]{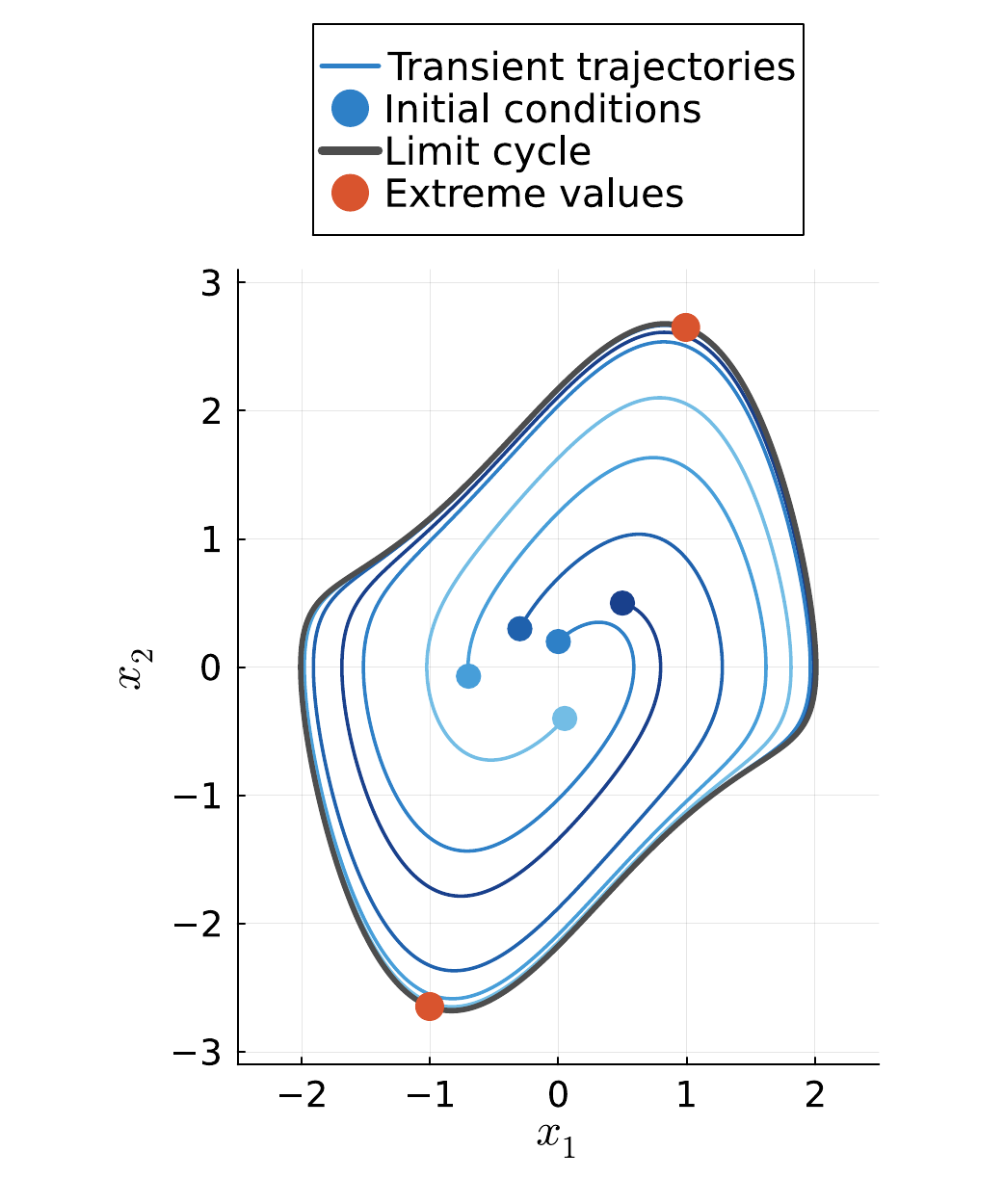}
    \caption{Trajectories of the Van der Pol oscillator approaching its stable limit cycle}
    \label{fig:vanderpol-trajectory}
\end{subfigure}

\caption{Three qualitatively different types of extreme behaviour: finite peak on the transient trajectory; finite-time blow-up; asymptotic approach to an extreme value on a stable limit cycle}
\label{fig:illustrative-trajectories}

\end{figure}

The blow-up and asymptotic scenarios are revisited in
Section~\ref{sec:examples}, where we apply the proposed method to compute
\chat{numerical estimates of the corresponding extreme values.}

\section{Compactification}\label{sec:compactification_reparametrization}
In this section, we first compactify the state space, which results in a non-polynomial dynamics of the system. To regain the polynomial property, we then use an appropriate time reparametrization.
\subsection{State compactification}
\label{sec_compactification}
To compactify the state space, we employ a compactification technique, as described in e.g. \cite{elias_critical_2006, matsue_blow-up_2018, takayasu_numerical_2017} for an autonomous system of polynomial equations. In the context of polynomial optimization, this corresponds to homogenization (or partial homogenization), which is described in \cite{huang_homogenization_2022, nie_moment_2023, unbounded_ctrl}. This compactification corresponds to lifting the phase plane $\R^n$ to a hemisphere in $\R^{n+1}$. This is sometimes called Poincaré compactification.  We define the following change of coordinates
\begin{equation}
\label{state_coordinate_change}
    x_i = \frac{z_i}{w} \quad \text{for all } i = 1,\ldots, n,
\end{equation}
where
\begin{equation}
\label{state_coordinate_constraints}
    (z_1,\ldots,z_n, w) = (z, w) \in \mathcal B \defeq \Set{(z, w) \in \R^{n+1}}{\sum_{i=1}^nz_i^2 + w^2 = 1, \quad w \geq 0}.
\end{equation}
The set $\mathcal B$ is called the Poincaré hemisphere.

\begin{proposition}
\label{prop_bijection}
The mapping
\begin{align}
\label{mapping_g}
    h: \overline{\mathbb{R}^n} \to \mathcal B,
\end{align}
defined for \(x \in \mathbb{R}^n\) by 
\begin{align*}
    x &\mapsto (z, w) = \left( \frac{x_1}{\sqrt{1+\sum_{i=1}^nx_i^2}},\ldots, \frac{x_n}{\sqrt{1+\sum_{i=1}^nx_i^2}}, \frac{1}{\sqrt{1+\sum_{i=1}^nx_i^2}} \right)
\end{align*}
and defined by
\[
    h(\infty \cdot a) = (a,0), \quad \|a\| = 1,
\]
is a homeomorphism.
\end{proposition}

\begin{proof}
The mapping $h$ clearly satisfies (\ref{state_coordinate_change}) - (\ref{state_coordinate_constraints}) and for every $z, w$, $w > 0$ we have a pre-image $x$ defined by (\ref{state_coordinate_change}). The pre-image of $\left(z, 0 \right)$ is $z\cdot\infty$ ($z$ describes the direction in which the original vector $x$ reaches infinity) according to the definition. Therefore, $h$ is a surjection. Whenever $h(x) = h(y)$, that means
\begin{align*}
    &\left( \frac{x_1}{\sqrt{1+\sum_{i=1}^nx_i^2}},\ldots, \frac{x_n}{\sqrt{1+\sum_{i=1}^nx_i^2}}, \frac{1}{\sqrt{1+\sum_{i=1}^nx_i^2}} \right)\\ =& \left( \frac{y_1}{\sqrt{1+\sum_{i=1}^ny_i^2}},\ldots, \frac{y_n}{\sqrt{1+\sum_{i=1}^ny_i^2}}, \frac{1}{\sqrt{1+\sum_{i=1}^ny_i^2}} \right),
\end{align*}
we have $x = y$. Indeed, the equality of the last component implies that $\sum_{i=1}^nx_i^2 = \sum_{i=1}^ny_i^2$, then from the equality of the $i^{\text{th}}$ component, it must hold that $x_i = y_i$ for every $i = 1,\ldots, n$. This concludes that $h$ is an injection. Since $h$ is continuous and has an inverse, $h^{-1}: (z, w)\mapsto \frac{z}{w}$, with $(z, 0) \mapsto z\cdot \infty$, which is also continuous, we conclude that $h$ is a homeomorphism.
\end{proof}

However, since the unbounded variable is the state, which has its dynamics, we must also modify the dynamics accordingly. That is, we have to find an expression for the dynamics of the new variables $(z, w)$. Using the chain rule, we have
\begin{align}
\label{w_der}
    \dot{w}(t) &= \frac{d}{dt}\left( \left(1+\sum_{i=1}^nx_i^2\right)^{-\frac{1}{2}} \right) = -w^{2}\sum_{i=1}^nz_if_i\left(\frac{z}{w}\right),\\
    \label{z_der}
    \dot{z_i}(t) &= \frac{d}{dt}\left(wx_i\right) =wf_i\left(\frac{z}{w}\right) - z_iw\sum_{j = 1}^n z_jf_j\left(\frac{z}{w}\right)
\end{align}

We label the new dynamics $(\dot z_1, \ldots, \dot z_n, \dot w) \defeq f_{\text{b}}(z, w)$. Notice that while the original dynamics $f$ was polynomial, the dynamics in homogeneous coordinates $f_{\text{b}}$ has rational terms.

Additionally, we also need the description of the set of all possible initial conditions $\mathcal X_0$ in homogenized variables $(z, w)$. Clearly, we can take the set 
\begin{equation*}
\mathcal Z_0 \defeq \ks{h(\mathcal X_0)},
\end{equation*}
where $h$ is the homeomorphism of Proposition \ref{prop_bijection}, but its description is not polynomial.
\begin{definition}[Homogenization of a polynomial]
    Given a polynomial $p(x)$ of a degree $d$, we denote its homogenization $\tilde p(z, w) \defeq w^dp\left(\frac{z}{w}\right)$.
\end{definition}
Instead of the set $\mathcal Z_0$, we take a basic semialgebraic set defined as 
\begin{equation}
\label{eq:initial_set_compact}
    \mathcal Y_0 \defeq \Set{(z, w) \in \mathcal B}{\tilde g_i\left(z, w\right) \geq 0, \quad i = 1,\ldots, N}.
\end{equation}
 We can assume the following assumption without loss of generality, as the initial set $\mathcal X_0$ is assumed to be compact.

\begin{assumption}\label{ass:ball_constraint}
    We assume that the ball constraint $K - \|x\|^2 \geq 0$ is in the definition of the set $\mathcal X_0$.
\end{assumption}
The relation between the sets $\mathcal Z_0$ and $\mathcal Y_0$ is explained by the following proposition.

\begin{proposition}\label{prop_initialsets_equivalence}
    Under Assumption \ref{ass:ball_constraint} it holds that \((\bar z, \bar w) \in \mathcal Z_0\) if and only if \((\bar z, \bar w) \in \mathcal Y_0.\)
\end{proposition}

\begin{proof}
    The mapping $x \mapsto (z, w)$, with $x = \frac{z}{w}$ and $(z, w) \in \mathcal B$, is a homeomorphism between $\mathcal X_0$ and $\mathcal Z_0$ in the homogenized coordinates. Since $\mathcal X_0$ is contained in the ball $\{x \in \mathbb{R}^n \mid \|x\|^2 \leq K\}$, we can derive a lower bound on $w$. For any $x \in \mathcal X_0$, the homogenization gives
    \[
        w = \frac{1}{\sqrt{1 + \|x\|^2}} \geq \frac{1}{\sqrt{1 + K}} > 0.
    \]

    $(\Rightarrow)$ Suppose $(\bar z, \bar w) \in \mathcal Z_0$, i.e., $(\bar z, \bar w) \in \mathcal B$ and $g_i\left(\frac{\bar z}{\bar w}\right) \geq 0$ for all $i = 1, \ldots, N$. Since $\bar w > 0$, we can multiply each inequality by $\bar w^{d_i}$, where $d_i$ is the degree of $g_i$, preserving the inequality
    \[
        \bar w^{d_i} g_i\left(\frac{\bar z}{\bar w}\right) = \tilde g_i(\bar z, \bar w) \geq 0, \quad \forall i.
    \]
    Hence, $(\bar z, \bar w) \in \mathcal Y_0$.

    $(\Leftarrow)$ Suppose $(\bar z, \bar w) \in \mathcal Y_0$, i.e., $(\bar z, \bar w) \in \mathcal B$ and $\tilde g_i(\bar z, \bar w) \geq 0$ for all $i$. Then
    \[
        \bar w^{d_i} g_i\left(\frac{\bar z}{\bar w}\right) \geq 0, \quad \forall i.
    \]

    If $\bar w > 0$, dividing each inequality by $\bar w^{d_i}$ yields $g_i\left(\frac{\bar z}{\bar w}\right) \geq 0$, so $(\bar z, \bar w) \in \mathcal Z_0$.

    Now consider $\bar w = 0$. In this case, the homogenized ball constraint (derived from the original $K - \|x\|^2 \geq 0$) becomes
    \[
        K \bar w^2 - \|\bar z\|^2 \geq 0 \quad \Rightarrow \quad -\|\bar z\|^2 \geq 0 \quad \Rightarrow \quad \bar z = 0.
    \]
    But this is in contradiction with the constraint $\|\bar z\|^2 + \bar w^2 = 1$. Therefore, $\bar w \neq 0$ for any $(\bar z, \bar w) \in \mathcal Y_0$, and the implication follows.
\end{proof}

\begin{remark}
\label{remark_ball_constraint}
    When $\mathcal Z_0$ (or equivalently $\mathcal X_0$) is compact, we can add the ball constraint without loss of generality to enforce the correspondence of $\mathcal Z_0$ and $\mathcal Y_0$. In general, if $\mathcal Z_0$ is allowed to be noncompact, the ball constraint obviously cannot be enforced and in order to obtain correspondence of these two sets, one needs a property called closedness at infinity; see for example \cite{huang_homogenization_2022}.
\end{remark}

The following example illustrates a situation where the correspondence of sets $\mathcal Z_0$ and $\mathcal Y_0$ fails without the ball constraint.

\begin{example}
    The set 
    \[
    \mathcal X_0 = \Set{x \in \R^2}{x_2^2 - x_1 + 1 \geq 0, x_1\geq 0, x_2 \geq 0, 1-x_2 \geq 0}
    \]
    is compact and does not contain the ball constraint. When we investigate its homogenization
    \[
    \mathcal Y_0 = \Set{(z, w) \in \mathcal B}{z_2^2 - z_1w + w^2 \geq 0,z_1 \geq 0, z_2 \geq 0, w - z_2 \geq 0},
    \]
    we come to the conclusion that $(1, 0, 0) \in \mathcal Y_0$, but it corresponds to $(\infty, 0) \notin \mathcal X_0$ in the original coordinates. Hence, $\mathcal X_0$ (and thus $\mathcal Z_0$) does not correspond to $\mathcal Y_0$.
\end{example}

\begin{remark}
    For the one-dimensional case (i.e, $n = 1$), it can be shown that the correspondence of $\mathcal Y_0$ and $\mathcal Z_0$ holds even without the ball constraint.
\end{remark}

Due to the reasoning above, we can take the set $\mathcal Y_0$ as the representation of the set $\mathcal X_0$ in the new variables $(z, w)$ for the dynamical system (\ref{w_der}) - (\ref{z_der}) as long as Assumption \ref{ass:ball_constraint} is satisfied. Notice that the dynamical system (\ref{w_der}) - (\ref{z_der}) \ks{in the interior of $\mathcal B$} is topologically equivalent\footnote{two systems are topologically equivalent, provided that there is a homeomorphism mapping trajectories of the first system to trajectories of the second system with the same direction of time} (see \cite{kuznetsov_2023}) to the original dynamical system $\dot{x}(t) = f(x)$ (according to Proposition \ref{prop_bijection}, the mapping $x \mapsto (z, w)$ is a homeomorphism). 

\subsection{Time reparameterization}
By construction, the compactified system \eqref{w_der}-\eqref{z_der} is not polynomial. To recover a polynomial property of the vector field $f_{\mathrm b}$, we use the following time reparameterization, as investigated in \cite{matsue_blow-up_2018} \[d\tau = \frac{1}{w^{d-1}}dt.\]
Observe that $t = 0$ corresponds to $\tau = 0$ and a time interval $[0, t_1)$ corresponds to a time interval $[0, \int_0^{t_1}\frac{1}{w^{d-1}}dt)$ in the new time variable $\tau$. After the time reparameterization, we arrive at the following dynamical system
\begin{align}
    \label{z_der_poly}
    \dot{z_i}(\tau) & =w^{d}f_i\left(\frac{z}{w}\right) - z_iw^{d}\sum_{j = 1}^n z_jf_j\left(\frac{z}{w}\right)\\
    \label{w_der_poly}
    \dot{w}(\tau) &= -w^{d+1}\sum_{i=1}^nz_if_i\left(\frac{z}{w}\right).
\end{align}
We can see that the system \eqref{z_der_poly}-\eqref{w_der_poly} is polynomial and it is equivalent to the system \eqref{w_der}-\eqref{z_der}, since the trajectories of these two systems are the same, with the only difference being the speed of the solution along the trajectories. We label the system $(\dot z_1, \ldots, \dot z_n, \dot w) =: f_{\text{p}}$.

\begin{proposition}
    If the solution of the original system in problem \eqref{problem_state_original} blows up at time $t_0 > 0$, the corresponding blow-up time in system \eqref{z_der_poly}-\eqref{w_der_poly} is $\tau_0 = \infty$.
\end{proposition}
\begin{proof}
\chat{Let
\[
\widetilde f_i(z,w):=w^d f_i(z/w),
\qquad
Q(z,w):=\sum_{i=1}^n z_i\widetilde f_i(z,w).
\]
Then $\widetilde f=(\widetilde f_1,\ldots,\widetilde f_n)$ is polynomial and
the reparameterized dynamics can be written as
\[
\dot z=\widetilde f(z,w)-zQ(z,w),
\qquad
\dot w=-wQ(z,w).
\]
Set
\[
H(z,w):=\|z\|^2+w^2-1.
\]
Along every local solution,
\[
\frac{d}{d\tau}H(z(\tau),w(\tau))
=-2Q(z(\tau),w(\tau))H(z(\tau),w(\tau)).
\]
Therefore $H(z(0),w(0))=0$ implies $H(z(\tau),w(\tau))=0$ throughout
the interval of existence. Moreover,
\[
w(\tau)
=w(0)\exp\left(-\int_0^\tau Q(z(u),w(u))\,du\right).
\]
Thus $w(0)\geq0$ implies $w(\tau)\geq0$, and $w(0)>0$ implies
$w(\tau)>0$ at every finite time. This proves that the hemisphere
$\mathcal B$ is invariant.

Since $Q$ is continuous and $\mathcal B$ is compact, there exists $M>0$
such that $|Q(z,w)|\leq M$ on $\mathcal B$. Hence
\[
w(\tau)\geq w(0)e^{-M\tau}>0
\qquad\text{for every finite }\tau
\]
whenever $w(0)>0$.

Now suppose that the solution of the original system blows up at time
$t_0>0$, so that $\|x(t)\|\to\infty$ as $t\uparrow t_0$. By the
compactification relation $x=z/w$, this is equivalent to $w\to0$ along
the corresponding compactified trajectory. Since $w$ cannot reach zero
at a finite value of $\tau$, this can occur only as $\tau\to\infty$.
Therefore the corresponding blow-up time for
\eqref{z_der_poly}--\eqref{w_der_poly} is $\tau_0=\infty$.}
\end{proof}

\chat{Because $f_{\mathrm p}$ is polynomial, it is locally Lipschitz on
$\mathbb R^{n+1}$. Every solution starting in $\mathcal B$ remains in the
compact invariant set $\mathcal B$, and hence extends uniquely to all
$\tau\in[0,\infty)$.}


For the rest of the work, we will use the following notation $y = (z, w)$ and  the system in the homogeneous coordinates after the time reparametrization will be denoted as
\begin{equation}\label{eq:reparametrized_compact_system}
    \dot y = f_{\text p}(y),\qquad y(0) \in \mathcal Y_0.
\end{equation}

\section{Occupation measure formulation}
\label{sec_occupation_measures}

 In this section, we consider a general dynamical system 
\begin{equation}\label{general_system}
    \dot x(t) = f(x(t)), \quad x(0)\in \mathcal X_0,
\end{equation}
where $f:\mathcal X \to\R^n$. The sets $\mathcal X$, $\mathcal X_0$ satisfy $\mathcal X_0 \subset \mathcal X \subset \R^n$ with $\mathcal X_0$ being compact and $\X$ closed. The set $\X$ is not required to be compact; in particular, $\X$ could be equal to all of $\mathbb{R}^n$. The set $\X$ will play the role of a constraint set if it is not invariant under the dynamics or that of the state-space itself if it is invariant. Either way, we suppose that $f$ is  globally Lipschitz on $\mathcal X$. This implies that any solution to~\eqref{general_system} that stays in $\X$ is continuously differentiable. 

We will use the following definition of admissible trajectory and occupation measure. 


\begin{definition}[Admissible trajectory]
Let $\bar t\in[0,\infty]$ and define
\[
I_{\bar t}:=
\begin{cases}
[0,\bar t], & \bar t<\infty,\\
[0,\infty), & \bar t=\infty.
\end{cases}
\]
We say that a continuously differentiable curve
\[
x:I_{\bar t}\to\mathbb R^n
\]
is an admissible trajectory of the system~\eqref{general_system}
with stopping time $\bar t$ if
\[
x(0)\in\mathcal X_0,
\qquad
x(t)\in\mathcal X,\qquad \dot{x}(t) = f(x(t))
\quad\text{for all }t\in I_{\bar t}.
\]
\end{definition}

\begin{definition}[Occupation, terminal and initial measures]
    Let $\bar t \in [0, \infty]$. We define the occupation measure $d\mu\left(t,x\right)$ associated with an admissible trajectory $x(\cdot) $ as the pushforward of the Lebesgue measure on $[0,\bar t)$ by the mapping $t\mapsto (t, x(t))$. \ks{That is, given a measurable, bounded and compactly supported function $\phi: [0, \infty)\times \X \to \R$, it is defined by relation
    \begin{equation*}
    \int_{[0, \infty)\times \mathcal X}\phi(t, x)d\mu(t,x) = \int_0^{\bar t}\phi(t, x(t))dt,
\end{equation*}
where on the left side, the variables $x, t$ are integrated variables, whereas on the right side, $t$ is the independent variable and $x(\cdot)$ is function depending on $t$.
    }
 Moreover, we define the terminal measure $\mu_{\mathrm T}$ corresponding to an admissible trajectory $x(t)$ with stopping time $\bar t$ as
 \ks{
 \begin{equation*}
\mu_{\mathrm T}
\defeq
\begin{cases}
\delta_{(\bar t,x(\bar t))}, & \bar t<\infty,\\
0, & \bar t=\infty.
\end{cases}
\end{equation*}
}
 Furthermore, we say a measure $\mu_0$ is an initial measure, provided that
 \begin{equation*}
    \mu_0\left(dt,dx\right) \defeq \delta_{(0,x_0)},
\end{equation*}
for some $x_0\in \mathcal X_0$.
\end{definition}
Observe that integration of a function $\phi \in C_c(\Rp\times \mathcal X)$ with respect to the terminal and initial measures $\mu_{\mathrm T}$ with stopping time \ks{$\bar t < \infty$} and $\mu_0$, respectively, is
\[
\int_{\Rp\times \mathcal X}\phi(t, x)d\mu_{\mathrm{T}} = \phi(\bar t, x(\bar t)), \qquad \int_{\Rp\times \mathcal X_0}\phi(t, x)d\mu_0 = \phi(0, x(0))
\]
for some $x(0) \in \mathcal X_0$.

Note that $\mathrm{spt}\mu$, $\mathrm{spt}\mu_{\text{T}}$, $\mathrm{spt}\mu_0 \subset [0, \infty)\times \mathcal X$ and the occupation measures corresponding to different admissible trajectories have different time supports in general. This corresponds to trajectories being defined only on some (possibly different) subintervals of $\Rp$.

\subsection{Liouville's equation}
Consider an ensemble of admissible trajectories $x(\cdot)$ satisfying the ODE \eqref{general_system}. Each trajectory may terminate at its own finite or infinite stopping time. Let \ks{$\phi\in C_c^{1}([0,\infty)\times \R^n)$} be a smooth compactly supported test function. Along every admissible trajectory $x(\cdot)$ we have
\[
\frac{d}{dt}\,\phi\bigl(t,x\bigr)
=\frac{\partial \phi\bigl(t,x\bigr)}{\partial t}+\nabla_x\phi\bigl(t,x\bigr)\!\cdot\!f\bigl(x\bigr).
\]
Integrating from $t=0$ to the terminal time $\bar t(x(\cdot))$ and then averaging with respect to the ensemble
of trajectories gives
\begin{align}
\label{eq_liouville_standard}
&\int_{[0,\infty)\times \mathcal X}\left(\frac{\partial \phi\bigl(t,x\bigr)}{\partial t}+\nabla_x\phi\bigl(t,x\bigr)\cdot f\bigl(x\bigr)\right)d\mu\\
&+\int_{\{0\}\times \mathcal X_0} \phi(t,x)\,d\mu_0(x)
-\int_{[0,\infty)\times \mathcal X}\phi(t,x)\,d\mu_{\mathrm T}(t,x)=0.\notag
\end{align}

This is the weak form of the Liouville's (continuity) equation on the infinite horizon with a time-distributed terminal measure. 

In distributional notation the equation \eqref{eq_liouville_standard} becomes,
\begin{equation}\label{liouville_state}
\partial_t\mu+\operatorname{div}(f\,\mu)
=\mu_0-\mu_{\mathrm T}.
\end{equation}

Equation \eqref{liouville_state} is to be understood in the sense of integrating test functions, as it was presented above. The equation \eqref{liouville_state} serves as the dynamics constraint $\dot{x} = f(x)$ together with the initial condition $x(0)\in \mathcal X_0$ (which is incorporated in the \ks{support of the} initial measure $\mu_0$) and one of its key properties is that it is linear in $(\mu,\mu_0,\mu_{\mathrm T})$. The choice of $\mass\mu_0$ is up to us; we can take $\mass\mu_0 = 1$. We also have $\mass\mu_{\mathrm T} \le \mass\mu_0$. However, $\mass \mu$ can be infinite, because \ks{the time horizon is unbounded and its time marginal is dominated by the Lebesgue measure}. For more details, we refer the reader to \cite{diperna_ordinary_1989} or \cite{korda_convex_2014} and references therein.

The following result says that any triple $(\mu_0, \mu, \mu_{\mathrm T})$ satisfying \eqref{eq_liouville_standard} is either an occupation measure or a superposition of occupation measures. Since occupation measures represent admissible trajectories, this result says that \eqref{eq_liouville_standard} truly encodes the dynamics. The proof of this statement is postponed to the appendix.

\begin{definition}
Let \(\Sinf\) denote the stopped-curve space consisting of pairs
\((\tau,\gamma(\cdot))\), where either
\[
\tau\in[0,\infty),
\qquad
\gamma\in C^1([0,\tau];\R^n),
\]
or
\[
\tau=\infty,
\qquad
\gamma\in C^1(\Rp;\R^n).
\]
\ks{We identify every finite stopped curve \(\gamma\in C^1([0,\tau];\mathbb R^n)\) with its constant extension \(\widetilde \gamma\in C([0,\infty))\), and equip \(\mathcal S_\infty\) with the Borel structure induced from \([0,\infty]\times C([0,\infty))\), where $C([0, \infty))$ is endowed with the topology of uniform convergence on every compact interval.
}
\end{definition}

\begin{theorem}[Superposition principle]
\label{thm:superposition}
Suppose that $f$ is globally Lipschitz on $\X$ and let
\[
\mu_0\in\mathcal P(\{0\}\times\mathcal X_0),
\qquad
\mu,\mu_{\mathrm T}\in\M_+([0,\infty)\times\mathcal X)
\]
be positive Radon measures satisfying  the Liouville's equation \eqref{eq_liouville_standard}. Then there exists a probability measure
\[
\eta\in\mathcal P(\mathcal S_\infty)
\]
such that, for \(\eta\)-almost every \((\tau,\gamma(\cdot))\), the curve \(\gamma\) is an admissible trajectory of \eqref{general_system} with stopping time \(\tau\).

Moreover,
\[
\mu_0=(e_0)_\#\eta,
\]
\[
\mu_{\mathrm T}
=
(e_{\mathrm T})_\#
\bigl(\eta|_{\{\tau<\infty\}}\bigr),
\]
and, for every compactly supported bounded Borel function
\[
h:[0,\infty)\times\mathcal X\to\mathbb R
\]
we have
\[
\int_{[0,\infty)\times\mathcal X}
h(t,x)\,d\mu(t,x)
=
\int_{\mathcal S_\infty}
\left(
\int_0^\tau h(t,\gamma(t))\,dt
\right)
d\eta(\tau,\gamma(\cdot)),
\]
where the inner integral is understood over \([0,\infty)\) if \(\tau=\infty\). Here
\[
e_0(\tau,\gamma(\cdot))=(0,\gamma(0)),
\]
and
\[
e_{\mathrm T}(\tau,\gamma(\cdot))=(\tau,\gamma(\tau)),
\qquad \tau<\infty.
\]
\end{theorem}

\subsection{Discounted Liouville's equation}
\label{sec:discounted_liouville}
\chat{Inspired by \cite{bhatt_occupation_1996}, we introduce the discounted
measures
\[
\nu:=e^{-\alpha t}\mu,
\qquad
\nu_{\mathrm T}:=e^{-\alpha t}\mu_{\mathrm T},
\qquad \alpha>0.
\]
For $\phi\in C^1([0,\infty)\times\mathcal X)$, define
\[
\mathcal L_\alpha^t\phi(t,x)
:=
\frac{\partial\phi}{\partial t}(t,x)
+\nabla_x\phi(t,x)\cdot f(x)
-\alpha\phi(t,x)
\]
and introduce the measure-independent test-function space
\[
\mathscr D_\alpha^t
:=
\left\{
\phi\in C^1([0,\infty)\times\mathcal X):
\phi\in C_b([0,\infty)\times\mathcal X),\quad
\mathcal L_\alpha^t\phi\in C_b([0,\infty)\times\mathcal X)
\right\}.
\]
We say that positive Radon measures
\[
\mu_0\in\mathcal M_+(\{0\}\times\mathcal X_0),
\qquad
\nu,\nu_{\mathrm T}\in
\mathcal M_+([0,\infty)\times\mathcal X)
\]
satisfy the discounted Liouville equation if, for every
$\phi\in\mathscr D_\alpha^t$, each integral in the following identity is
finite and}
\begin{align}\label{eq:discounted_liouville}
    \int_{[0, \infty)\times \mathcal X}\left(\frac{\partial \phi}{\partial t} + \nabla_x \phi \cdot f - \alpha \phi\right)d\nu + \int_{\{0\}\times \mathcal X_0}\phi d\mu_0 - \int_{[0, \infty)\times \mathcal X}\phi d\nu_{\mathrm T} = 0.
\end{align}
\chat{The constant function $1$ belongs to $\mathscr D_\alpha^t$ and
$\mathcal L_\alpha^t1=-\alpha$. Consequently, testing
\eqref{eq:discounted_liouville} with $\phi = 1$ gives
\[
\alpha\,\mass\nu+\mass\nu_{\mathrm T}=\mass\mu_0.
\]
Thus $\nu$ and $\nu_{\mathrm T}$ have finite total mass whenever $\mu_0$
does. The additional sink term therefore accounts exactly for the mass of
trajectories that continue for an infinite time.}

\chat{\begin{proposition}\label{prop:discounted_equivalence}
Let $\alpha>0$ and suppose that the hypotheses of
Theorem~\ref{thm:superposition} hold. Then
\[
(\mu_0,\mu,\mu_{\mathrm T})
\longmapsto
(\mu_0,e^{-\alpha t}\mu,e^{-\alpha t}\mu_{\mathrm T})
\]
is a one-to-one correspondence between positive-measure solutions of the
Liouville equation~\eqref{eq_liouville_standard} with
$\mu_0\in\mathcal P(\{0\}\times\mathcal X_0)$ and positive-measure
solutions of the discounted Liouville equation
\eqref{eq:discounted_liouville} with the same initial measure. Its inverse is
\[
(\mu_0,\nu,\nu_{\mathrm T})
\longmapsto
(\mu_0,e^{\alpha t}\nu,e^{\alpha t}\nu_{\mathrm T}).
\]
\end{proposition}

\begin{proof}
Let $(\mu_0,\mu,\mu_{\mathrm T})$ satisfy
\eqref{eq_liouville_standard}. By Theorem~\ref{thm:superposition}, it has a
representation by a probability measure $\eta\in\mathcal P(\mathcal S_\infty)$.
\chat{The occupation identity in that theorem extends from compactly
supported bounded Borel functions to nonnegative Borel functions by
monotone convergence.}
The measures
\[
d\nu(t,x):=e^{-\alpha t}\,d\mu(t,x),
\qquad
d\nu_{\mathrm T}(t,x):=e^{-\alpha t}\,d\mu_{\mathrm T}(t,x)
\]
are finite, since
\[
\mass\nu
=
\int_{\mathcal S_\infty}\int_0^\tau e^{-\alpha t}\,dt\,d\eta
\leq\frac1\alpha,
\qquad
\mass\nu_{\mathrm T}
=
\int_{\{\tau<\infty\}}e^{-\alpha\tau}\,d\eta
\leq1.
\]
\chat{By applying the preceding monotone-convergence extension to the
positive and negative parts, the occupation identity may now also be used
for every Borel integrand that is absolutely integrable with respect to
$\mu$.}
Let $\phi\in\mathscr D_\alpha^t$. Along every represented trajectory,
\[
\frac{d}{dt}\left(e^{-\alpha t}\phi(t,x(t))\right)
=e^{-\alpha t}\mathcal L_\alpha^t\phi(t,x(t)).
\]
If $\tau=\infty$, the terminal term vanishes because $\phi$ is bounded.
Integrating this identity up to the stopping time, integrating with
respect to $\eta$, and using the occupation and terminal representations
gives \eqref{eq:discounted_liouville}. All its integrals
are finite because both $\phi$ and $\mathcal L_\alpha^t\phi$ are bounded.

Conversely, let $(\mu_0,\nu,\nu_{\mathrm T})$ satisfy
\eqref{eq:discounted_liouville}, and define
\[
d\mu(t,x):=e^{\alpha t}\,d\nu(t,x),
\qquad
d\mu_{\mathrm T}(t,x):=e^{\alpha t}\,d\nu_{\mathrm T}(t,x).
\]
These are Radon measures because the multiplier $e^{\alpha t}$ is bounded
on compact time intervals. For any
$\psi\in C_c^1([0,\infty)\times\mathcal X)$, the function
\[
\phi(t,x):=e^{\alpha t}\psi(t,x)
\]
belongs to $\mathscr D_\alpha^t$ and satisfies
\[
\mathcal L_\alpha^t\phi
=e^{\alpha t}\left(\frac{\partial\psi}{\partial t}
+\nabla_x\psi\cdot f\right).
\]
Substitution in \eqref{eq:discounted_liouville} gives
\eqref{eq_liouville_standard}. The two measure transformations are plainly
mutually inverse.
\end{proof}

The following result is an immediate consequence of
Theorem~\ref{thm:superposition} and
Proposition~\ref{prop:discounted_equivalence}.

\begin{corollary}[Superposition principle for the discounted equation]
\label{cor:superposition_discounted}
\chat{Suppose that the hypotheses of Theorem~\ref{thm:superposition}
hold.} Let $\alpha>0$ and let
\[
\mu_0\in\mathcal P(\{0\}\times\mathcal X_0),
\qquad
\nu,\nu_{\mathrm T}\in\mathcal M_+([0,\infty)\times\mathcal X)
\]
satisfy \eqref{eq:discounted_liouville}. Then there exists
$\eta\in\mathcal P(\mathcal S_\infty)$ such that, for $\eta$-almost every
$(\tau,\gamma)$, the curve $\gamma$ is an admissible trajectory of
\eqref{general_system} with stopping time $\tau$. Moreover, for every
bounded Borel function $h:[0,\infty)\times\mathcal X\to\mathbb R$,
\[
\int_{[0,\infty)\times\mathcal X}h(t,x)\,d\nu(t,x)
=
\int_{\mathcal S_\infty}
\left(\int_0^\tau e^{-\alpha t}h(t,\gamma(t))\,dt\right)
d\eta(\tau,\gamma),
\]
where the inner integral is over $[0,\infty)$ if $\tau=\infty$, and
\[
\int_{[0,\infty)\times\mathcal X}h(t,x)\,d\nu_{\mathrm T}(t,x)
=
\int_{\{\tau<\infty\}}
e^{-\alpha\tau}h(\tau,\gamma(\tau))\,d\eta(\tau,\gamma).
\]
Finally,
\[
\mu_0=(\widetilde e_0)_\#\eta,
\qquad
\widetilde e_0(\tau,\gamma):=(0,\gamma(0)).
\]
\end{corollary}

\begin{proof}
Apply Proposition~\ref{prop:discounted_equivalence} to recover a solution
$(\mu_0,\mu,\mu_{\mathrm T})$ of
\eqref{eq_liouville_standard}, and then apply
Theorem~\ref{thm:superposition}. Multiplying the resulting occupation and
terminal measures by $e^{-\alpha t}$ gives the displayed identities. The
identities extend from compactly supported bounded Borel functions to all
bounded Borel functions because the two discounted measures are finite.
\end{proof}}

\subsection{Discounted Liouville's equation with artificial clock variable}
\label{sec:discounted_clock}

In this section, we combine the time-discounting approach of
Section~\ref{sec:discounted_liouville} with the compactification of the
time axis via an artificial clock variable $s$. Consider the augmented system
\begin{align}
\label{augmented_system}
    \dot x(t) &= f(x(t)),
    \qquad x(t)\in\mathcal X,
    \qquad x(0)\in\mathcal X_0,\\
    \dot s(t) &= -\alpha s(t),
    \qquad s(t)\in[0,1],
    \qquad s(0)=1,\notag
\end{align}
for some \(\alpha>0\). The solution of the clock dynamics is
\[
s(t)=e^{-\alpha t}.
\]
Hence, the map
\[
t\mapsto s(t)=e^{-\alpha t}
\]
is a bijection from \([0,\infty)\) onto \((0,1]\), with inverse
\[
t(s)=-\frac{1}{\alpha}\log s.
\]
By setting \(s(\infty)=0\), this map extends to a homeomorphism from the
one-point compactification \([0,\infty]\) onto \([0,1]\).
Since \(t\) is uniquely determined by \(s\) for every finite time, the discounted occupation and terminal measures of the augmented system (introduced in the previous section) are completely determined by their marginals on the variables \((s,x)\). Thus, it is sufficient to use test functions depending only on $(s, x)$ to represent the dynamics via the discounted Liouville's equation \eqref{eq:discounted_liouville}.

\chat{For $\phi\in C^1([0,1]\times\mathcal X)$, define the discounted
clock Liouville operator
\[
\mathcal L_\alpha^s\phi(s,x)
:=
-\alpha s\frac{\partial\phi}{\partial s}(s,x)
+\nabla_x\phi(s,x)\cdot f(x)
-\alpha\phi(s,x)
\]
and the test-function space
\[
\mathscr D_\alpha^s
:=
\left\{
\phi\in C^1([0,1]\times\mathcal X):
\phi\in C_b([0,1]\times\mathcal X),\quad
\mathcal L_\alpha^s\phi\in C_b([0,1]\times\mathcal X)
\right\}.
\]
We say that positive Radon measures
\[
\mu_0\in\mathcal M_+(\{1\}\times\mathcal X_0),
\qquad
\nu,\nu_{\mathrm T}\in\mathcal M_+([0,1]\times\mathcal X)
\]
satisfy the discounted clock Liouville equation if, for every
$\phi\in\mathscr D_\alpha^s$, each integral below is finite and}
\begin{align}
\label{eq:discounted_clock_liouville}
&\int_{[0,1]\times\mathcal X}
\left(
-\alpha s\,\frac{\partial\phi(s,x)}{\partial s}
+\nabla_x\phi(s,x)\cdot f(x)
-\alpha\phi(s,x)
\right)d\nu(s,x)
\\
&\quad
+\int_{\{1\}\times\mathcal X_0}
\phi(s,x)\,d\mu_0(s,x)
-\int_{[0,1]\times\mathcal X}
\phi(s,x)\,d\nu_{\mathrm T}(s,x)
=0.
\notag
\end{align}

\chat{The function $1$ belongs to $\mathscr D_\alpha^s$ and
$\mathcal L_\alpha^s1=-\alpha$. Testing
\eqref{eq:discounted_clock_liouville} with $\phi = 1$ gives
\[
\alpha\,\mass\nu+\mass\nu_{\mathrm T}=\mass\mu_0.
\]
Consequently, $\nu$ and $\nu_{\mathrm T}$ have finite total mass whenever
$\mu_0$ does. }

\chat{\begin{proposition}\label{prop:discounted_clock_equivalence}
Let $\alpha>0$ and define
\[
g:[0,\infty)\times\mathcal X\longrightarrow(0,1]\times\mathcal X,
\qquad
g(t,x):=(e^{-\alpha t},x).
\]
Then
\[
(\mu_0,\nu,\nu_{\mathrm T})
\longmapsto
(\widehat\mu_0,\widehat\nu,\widehat\nu_{\mathrm T})
:=(g_\#\mu_0,g_\#\nu,g_\#\nu_{\mathrm T})
\]
is a one-to-one correspondence between solutions of the discounted
Liouville equation~\eqref{eq:discounted_liouville} and solutions of the
discounted clock Liouville equation
\eqref{eq:discounted_clock_liouville}; no compactness of $\mathcal X$ is
required. Every solution of the latter
equation satisfies
\[
\widehat\nu(\{0\}\times\mathcal X)
=
\widehat\nu_{\mathrm T}(\{0\}\times\mathcal X)
=0.
\]
Consequently, the inverse correspondence is
\[
\mu_0=(g^{-1})_\#\widehat\mu_0,
\qquad
\nu=(g^{-1})_\#
\left(\left.\widehat\nu\right|_{(0,1]\times\mathcal X}\right),
\qquad
\nu_{\mathrm T}=(g^{-1})_\#
\left(\left.\widehat\nu_{\mathrm T}\right|_{(0,1]\times\mathcal X}\right),
\]
where $g^{-1}(s,x)=(-\alpha^{-1}\log s,x)$ for $s>0$.
\end{proposition}

\begin{proof}
Let $(\mu_0,\nu,\nu_{\mathrm T})$ satisfy
\eqref{eq:discounted_liouville}, and define
\[
\widehat\mu_0:=g_\#\mu_0,
\qquad
\widehat\nu:=g_\#\nu,
\qquad
\widehat\nu_{\mathrm T}:=g_\#\nu_{\mathrm T}.
\]
These are finite Radon measures on their stated clock-space domains.
For $\phi\in\mathscr D_\alpha^s$, set
\[
\psi(t,x):=\phi(e^{-\alpha t},x).
\]
The function $\psi$ belongs to $\mathscr D_\alpha^t$, because it is bounded
and the chain rule gives
\[
\mathcal L_\alpha^t\psi(t,x)
=
\mathcal L_\alpha^s\phi(e^{-\alpha t},x),
\]
which is bounded. Applying \eqref{eq:discounted_liouville} to $\psi$ and
using the definition of pushforward gives
\eqref{eq:discounted_clock_liouville} for $\phi$. Thus
$(\widehat\mu_0,\widehat\nu,\widehat\nu_{\mathrm T})$ satisfies the
discounted clock equation.

Conversely, let
$(\widehat\mu_0,\widehat\nu,\widehat\nu_{\mathrm T})$ satisfy
\eqref{eq:discounted_clock_liouville}. Testing with $1$ shows that
$\widehat\nu$ and $\widehat\nu_{\mathrm T}$ are finite. For every
\claude{$m\in\mathbb N$}, the function
\[
\claude{\phi_m(s,x):=e^{-ms}}
\]
belongs to $\mathscr D_\alpha^s$, and
\[
\claude{\mathcal L_\alpha^s\phi_m(s,x)
=
\alpha(ms-1)e^{-ms}.}
\]
Consequently,
\[
\claude{\alpha\int_{[0,1]\times\mathcal X}(ms-1)e^{-ms}
\,d\widehat\nu(s,x)
+e^{-m}\mass\widehat\mu_0
-\int_{[0,1]\times\mathcal X}e^{-ms}
\,d\widehat\nu_{\mathrm T}(s,x)
=0.}
\]
As \claude{$m\to\infty$},
\[
\claude{(ms-1)e^{-ms}\longrightarrow-\mathbf 1_{\{0\}}(s),
\qquad
e^{-ms}\longrightarrow\mathbf 1_{\{0\}}(s),}
\]
and
\[
\claude{|(ms-1)e^{-ms}|\leq1,
\qquad
0\leq e^{-ms}\leq1.}
\]
Dominated convergence therefore gives
\[
-\alpha\widehat\nu(\{0\}\times\mathcal X)
-\widehat\nu_{\mathrm T}(\{0\}\times\mathcal X)
=0.
\]
Positivity of the measures implies
\begin{equation}\label{eq:s_zero_masses}
\widehat\nu(\{0\}\times\mathcal X)
=
\widehat\nu_{\mathrm T}(\{0\}\times\mathcal X)
=0.
\end{equation}

The inverse map
\[
g^{-1}(s,x)
=
\left(-\frac1\alpha\log s,x\right)
\]
is defined on $(0,1]\times\mathcal X$. In view of
\eqref{eq:s_zero_masses}, define
\[
\mu_0:=(g^{-1})_\#\widehat\mu_0,
\]
\[
\nu
:=
(g^{-1})_\#\left(
\left.\widehat\nu\right|_{(0,1]\times\mathcal X}
\right),
\qquad
\nu_{\mathrm T}
:=
(g^{-1})_\#\left(
\left.\widehat\nu_{\mathrm T}\right|_{(0,1]\times\mathcal X}
\right).
\]
These inverse images are again finite Radon measures.

It remains to check the discounted equation for every test function in
$\mathscr D_\alpha^t$. Let $\psi\in\mathscr D_\alpha^t$. Choose
$\rho\in C_c^1([0,\infty))$ such that
\[
0\leq\rho\leq1,
\qquad
\rho=1\ \text{on }[0,1],
\qquad
\rho=0\ \text{on }[2,\infty),
\]
and set
\[
\rho_R(t):=\rho(t/R),
\qquad
\psi_R(t,x):=\rho_R(t)\psi(t,x).
\]
for $R\geq1$.
Define
\[
\phi_R(s,x)
:=
\begin{cases}
\displaystyle
\psi_R\left(-\frac1\alpha\log s,x\right),&s>0,\\[1ex]
0,&s=0.
\end{cases}
\]
Because $\psi_R$ vanishes for sufficiently large $t$, the function $\phi_R$
vanishes in a neighborhood of $s=0$. Moreover,
\[
\mathcal L_\alpha^s\phi_R(s,x)
=
\mathcal L_\alpha^t\psi_R
\left(-\frac1\alpha\log s,x\right)
\qquad(s>0),
\]
so $\phi_R\in\mathscr D_\alpha^s$. Applying
\eqref{eq:discounted_clock_liouville} to $\phi_R$ and using the inverse
pushforwards shows that \eqref{eq:discounted_liouville} holds with
$\psi_R$ in place of $\psi$.

Finally,
\[
\mathcal L_\alpha^t\psi_R
=
\rho_R\mathcal L_\alpha^t\psi+\rho_R'\psi.
\]
The first term converges pointwise to $\mathcal L_\alpha^t\psi$ and is
uniformly bounded. The cutoff error satisfies
\[
\left|\int\rho_R'(t)\psi(t,x)\,d\nu(t,x)\right|
\leq
\frac{\|\rho'\|_\infty\|\psi\|_\infty}{R}\mass\nu
\longrightarrow0.
\]
Dominated convergence in the remaining terms therefore gives
\eqref{eq:discounted_liouville} for $\psi$. The two pushforward
constructions are mutually inverse because $g$ and $g^{-1}$ are inverse on
$[0,\infty)\times\mathcal X$ and $(0,1]\times\mathcal X$, respectively,
and the clock-space measures give no mass to $s=0$.
\end{proof}}

\chat{The following result follows from
Propositions~\ref{prop:discounted_equivalence} and
\ref{prop:discounted_clock_equivalence} together with
Theorem~\ref{thm:superposition}.

\begin{corollary}[Superposition principle for the discounted clock equation]
\label{cor:superposition_discounted_clock}
\chat{Suppose that the hypotheses of Theorem~\ref{thm:superposition}
hold.} Let $\alpha>0$ and let
\[
\mu_0\in\mathcal P(\{1\}\times\mathcal X_0),
\qquad
\nu,\nu_{\mathrm T}\in\mathcal M_+([0,1]\times\mathcal X)
\]
satisfy \eqref{eq:discounted_clock_liouville}. Then there exists
$\eta\in\mathcal P(\mathcal S_\infty)$ such that, for $\eta$-almost every
$(\tau,\gamma)$, the curve $\gamma$ is an admissible trajectory of
\eqref{general_system} with stopping time $\tau$. Moreover, for every
bounded Borel function $h:[0,1]\times\mathcal X\to\mathbb R$,
\[
\int_{[0,1]\times\mathcal X}h(s,x)\,d\nu(s,x)
=
\int_{\mathcal S_\infty}
\left(
\int_0^\tau e^{-\alpha t}
h(e^{-\alpha t},\gamma(t))\,dt
\right)d\eta(\tau,\gamma),
\]
where the inner integral is over $[0,\infty)$ if $\tau=\infty$, and
\[
\int_{[0,1]\times\mathcal X}h(s,x)\,d\nu_{\mathrm T}(s,x)
=
\int_{\{\tau<\infty\}}e^{-\alpha\tau}
h(e^{-\alpha\tau},\gamma(\tau))\,d\eta(\tau,\gamma).
\]
Finally,
\[
\mu_0=(\widetilde e_0)_\#\eta,
\qquad
\widetilde e_0(\tau,\gamma):=(1,\gamma(0)).
\]
\end{corollary}

\begin{proof}
Apply Proposition~\ref{prop:discounted_clock_equivalence} to obtain a
solution of \eqref{eq:discounted_liouville}, and then apply
Corollary~\ref{cor:superposition_discounted}. Pushing its occupation and
terminal representations forward under
$g(t,x)=(e^{-\alpha t},x)$ gives the displayed formulas.
\end{proof}

\begin{remark}\label{rem:polynomial_test_functions}
If $\mathcal X$ is compact, then every function in
$C^1([0,1]\times\mathcal X)$ and its image under
$\mathcal L_\alpha^s$ are bounded. Hence
\[
\mathscr D_\alpha^s=C^1([0,1]\times\mathcal X).
\]
In particular, polynomial test functions are admissible in the discounted
clock equation \eqref{eq:discounted_clock_liouville}.
\end{remark}}

\section{Getting the extreme value}\label{sec:peak_recovery}
It is standard (see e.g. \cite{miller_peak_2024}), that for extreme value estimation problems constrained by the standard (nondiscounted) Liouville's equation \eqref{eq_liouville_standard} on finite time horizon $T$, one can use the (linear) cost functional
\[
\sup \int_{[0, T]\times \mathcal X}\|x\|^2d\mu_{\mathrm T}
\]
to obtain the extreme value. However, this is no longer the case when the problem is constrained by the discounted Liouville's equation \eqref{eq:discounted_clock_liouville}, which we aim to use, since the standard Liouville's equation $\eqref{eq_liouville_standard}$ is not well suited for the moment-SOS hierarchy because of the infinite time horizon. The reason is that the terminal measure $\nu_{\mathrm T}$ is discounted, meaning that its mass is (exponentially) decreasing with time and the mass of $\nu_{\mathrm T}$ scales the value of the extreme value $\int_{[0, \infty)\times \mathcal X}\|x\|^2d\nu_{\mathrm T}$. For this reason, we develop a different technique.

\subsection{Reachability testing}
\label{sec:bisection}
In this section, we consider the compactified polynomial vector field after the time reparametrization \eqref{eq:reparametrized_compact_system}. Hence, the trajectories for this system lie on the hemisphere \[\mathcal B = \Set{y = (z, w)\in \R^{n+1}}{\|y\|^2 = 1, w \geq 0}\] defined in Section \ref{sec_compactification}. Note that for the extreme value quantity it holds 
\begin{equation}\label{eq:peak_quantity_equivalence}
    \frac{\|x\|^2}{1+ \|x\|^2} = \|z\|^2,
\end{equation}
where $y = (z, w)$.

To recover the extreme value despite the discounting of the terminal measure, we adopt a bisection  approach. The key idea is to avoid using the value of the discounted terminal cost itself and instead determine whether trajectories can reach a given level set of the state norm.

Let $r\in (0, 1)$ and denote by
\[
\mathcal B_r := \{y = (z, w)\in\mathbb R^{n+1} \mid \|z\|^2 \le r\} \cap\mathcal B
\]
the closed ball of squared radius $r$, and by $\overline{\mathcal B\setminus \mathcal B_r}$ its closed complement inside the compactified state space $\mathcal B$. We interpret $r$ as a candidate value for the extreme value.

Using the discounted Liouville's equation with the artificial clock variable
\eqref{eq:discounted_clock_liouville}, we ask whether there exists an admissible trajectory starting from $\mathcal Y_0 \subset \mathcal B$ whose state stays in $\mathcal B_r$ until it reaches the set $\overline{\mathcal B\setminus \mathcal B_r}$ at some time. In the measure-theoretic formulation, this corresponds to the existence of measures
\begin{align}
\label{eq:support_compact_bisection}
    \nu \in \mathcal M_+([0,1]\times \mathcal B_r), 
\qquad
\nu_{\mathrm T} \in \mathcal M_+([0,1]\times \overline{\mathcal B\setminus \mathcal B_r}),
\qquad
\mu_0 \in \mathcal M_+(\{1\}\times \mathcal Y_0),
\end{align}

satisfying the discounted clock Liouville's equation \eqref{eq:discounted_clock_liouville}. Observe that we restricted the supports of measures $\nu$ and $\nu_{\mathrm T}$. Condition $\nu \in \mathcal M_+([0,1]\times \mathcal B_r)$ enforces that the occupation measure only captures trajectories whose squared norm does not exceed $r$, while $\nu_{\mathrm T} \in \mathcal M_+([0,1]\times \overline{\mathcal B\setminus \mathcal B_r})$ assures the terminal measure is supported on states whose squared norm is at least $r$. \chat{Before imposing a normalization on $\mu_0$, the all-zero triple is always feasible, so feasibility with $\nu_{\mathrm T}=0$ has no trajectory-level meaning. After imposing $\mass\mu_0=1$ as in \eqref{eq:feas_for_bisection_compact}, any feasible solution with $\mass\nu_{\mathrm T}>0$ represents an ensemble containing at least one trajectory that reaches the level $\|z\|^2=r$ or exceeds it at a finite time.}

Hence, we would like to maximize $\mass \nu_{\mathrm T}$ to see if the problem is feasible with $\nu_{\mathrm T}$ of positive mass. The formal statement of the problem is
\begin{align}\label{eq:feas_for_bisection_compact}
    \beta(r) := &\sup_{\mu_0, \nu, \nu_{\mathrm T}} \mass \nu_{\mathrm T}\\
    \text{s.t.: }& \int_{[0,1]\times\mathcal B}
\left(
-\alpha s\,\frac{\partial \phi(s,y)}{\partial s}
+ \nabla_y\phi(s,y)\cdot f_{\text{p}}(y)
- \alpha \phi(s,y)
\right)
d\nu(s,y)\notag \\
&+ \int_{\{1\}\times\mathcal Y_0} \phi(s,y)\,d\mu_0(s,y)
- \int_{[0,1]\times\mathcal B} \phi(s,y)\,d\nu_{\mathrm T}(s,y)
= 0,\notag \\
&\quad \chat{\text{for all }\phi\in\mathscr D_\alpha^s,}\notag\\
&     \nu \in \mathcal M_+([0,1]\times \mathcal B_r), 
\quad
\nu_{\mathrm T} \in \mathcal M_+([0,1]\times \overline{\mathcal B\setminus \mathcal B_r}),
\quad
\mu_0 \in \mathcal M_+(\{1\}\times \mathcal Y_0),\notag \\&\quad \mass\mu_0 = 1.\notag
\end{align}

\mk{If the optimal value of \eqref{eq:feas_for_bisection_compact} is zero, then no admissible
trajectory attains $\|z(\tau)\|^2 \ge r$ at any finite time $\tau$. Hence
$\|z(\tau)\|^2 < r$ along every admissible trajectory, and therefore $r$ is an
upper bound on the extreme value, i.e.
\[
  \sup_{\tau \ge 0} \|z(\tau)\|^2 \le r .
\]
Note that this bound need not be strict: the supremum may equal $r$ when the
level $r$ is approached only asymptotically as $\tau \to \infty$ rather than
attained at a finite time.}


This leads to the following bisection procedure. Let $[\underline r,\overline r]$ be an interval known to contain the true extreme value. The lower bound for the bisection $\underline r$ can be determined as any reachable $r$ by the trajectories, that is, given any point $y_0 = (z_0, w_0)\in \mathcal Y_0$, we can choose $\underline r = \|z_0\|^2$ and the upper bound can be selected as $\overline r = 1$ because of the compactification. At each iteration, we test whether the optimal value $\beta(r)$ of \eqref{eq:feas_for_bisection_compact} is strictly positive for $r=(r_{\min}+r_{\max})/2$. If $\beta(r) > 0$, we set $r_{\min}\leftarrow r$; otherwise, we set $r_{\max}\leftarrow r$. Repeating this procedure yields convergence of $r_{\min}$ and $r_{\max}$ to the true extreme value. This convergence is more precisely described by the following result.

\begin{proposition}[Convergence of bisection to the extreme value]
\label{prop:bisection_convergence}
\chat{Fix $\alpha>0$ and let $\beta$ be defined by
\eqref{eq:feas_for_bisection_compact}. Assume
\[
0\leq\underline r<\overline r\leq1,
\qquad
\beta(\underline r)>0,
\qquad
\beta(\overline r)=0.
\]
Starting from
$(r_{\min}^0,r_{\max}^0)=(\underline r,\overline r)$, perform the exact
bisection update described above, and let
\[
r^\star := \sup \Set{r \in [\underline r, \overline r)}{\beta(r) > 0}.
\]
Then
\[
r_{\min}^j\uparrow r^\star,
\qquad
r_{\max}^j\downarrow r^\star.
\]
Moreover, $r^\star$ is the true extreme value and every $r_{\max}^j$ is
an upper bound:}
\[
\chat{\sup_{t\ge 0}\|z(t)\|^2 \le r_{\max}^j
\qquad\text{for every admissible trajectory }
y(\cdot)=(z(\cdot),w(\cdot))\text{ and every }j.}
\]
\end{proposition}

\begin{proof}
\chat{Fix $r\in(0,1)$} and denote by $\beta(r)$ the optimal value of \eqref{eq:feas_for_bisection_compact}.

First, we show that if $\beta(r)>0$, then the level $r$ is reachable. Assume $\beta(r)>0$. Then there exists a feasible triplet $(\nu,\nu_{\mathrm T},\mu_0)$ in
\eqref{eq:feas_for_bisection_compact} such that $\mass\,\nu_{\mathrm T}>0$ and
\[
\mathrm{spt} \nu \subset [0,1]\times \mathcal B_r,\qquad \mathrm{spt} \nu_{\mathrm T} \subset [0,1]\times \overline{\mathcal B\setminus \mathcal B_r}.
\]
By Corollary \ref{cor:superposition_discounted_clock}, $(\nu,\nu_{\mathrm T},\mu_0)$ admits a representation by an ensemble of admissible trajectories $y(\cdot)$ of the compactified system \chat{with finite or infinite stopping times}. The representation is such that $\nu$ is the occupation measure and $\nu_{\mathrm T}$ the terminal measure. Since $\mass\,\nu_{\mathrm T}>0$ and $\mathrm{spt} \nu_{\mathrm T} \subset [0,1]\times \overline{\mathcal B\setminus \mathcal B_r}$, there must exist
(at least one) trajectory in the ensemble whose terminal state lies in $\overline{\mathcal B\setminus \mathcal B_r}$. For this trajectory we have $\|z(\tau)\|^2\ge r$ at the stopping time, hence $r$ is reachable.

\ks{Conversely, suppose that an admissible trajectory reaches $\overline{\mathcal B\setminus \mathcal B_r}$ at a finite time. Stop the trajectory at its first hitting time $\tau$. Its discounted occupation measure is supported on $[0,1]\times\mathcal B_r$, its terminal measure is supported on $[0,1]\times\overline{\mathcal B\setminus \mathcal B_r}$, with
\(
\operatorname{mass}\nu_T=e^{-\alpha\tau}>0.
\)
Therefore $\beta(r)>0$.}

Now, if $r$ is not reachable by any admissible trajectory, then $\beta(r)=0$. Evidently, if $\beta(r)=0$, then no admissible trajectory can satisfy $\|z(\tau)\|^2\ge r$ at any finite time. Hence
\[
\sup_{\tau\ge 0}\|z(\tau)\|^2 \le r
\qquad\text{for all admissible trajectories }\chat{y(\cdot)},
\]
so $r$ is an upper bound on the true extreme value.
\ks{Consequently, \(r^\star\) coincides with the true extreme value,
\[
r^\star
=
\sup\left\{
\|z(t)\|^2
\;\middle|\;
y(\cdot)=(z(\cdot),w(\cdot)) \text{ is admissible},\ t\geq 0
\right\},
\]
since every level strictly below the right-hand side is reached at a finite time, while every level reached at a finite time is bounded by the right-hand side.}

By construction, the bisection procedure produces a nonincreasing sequence \chat{$(r_{\max}^j)_j$} such that \chat{$r_{\max}^j$} is an upper bound for every \chat{$j$}. This yields
\[
\sup_{\tau\ge 0}\|z(\tau)\|^2 \le \chat{r_{\max}^j}
\qquad\text{for all admissible trajectories }\chat{y(\cdot)},\ \text{for all }\chat{j}.
\]

Finally, by construction of the bisection updates we have \chat{$r_{\min}^j\le r_{\max}^j$ for every $j$} and
\[
\chat{r_{\max}^{j+1}-r_{\min}^{j+1}=\tfrac12\,(r_{\max}^j-r_{\min}^j)},
\]
hence \chat{$r_{\max}^j-r_{\min}^j=2^{-j}(r_{\max}^0-r_{\min}^0)\to0$}. Moreover, the update rule preserves the invariant \chat{$\beta(r_{\min}^j)>0$}, and \chat{$r_{\max}^j$} is an upper bound for all \chat{$j$}, so \chat{$r_{\min}^j\le r^\star\le r_{\max}^j$} for all \chat{$j$} by definition of $r^\star$. Taking limits gives \chat{$r_{\min}^j\uparrow r^\star$ and $r_{\max}^j\downarrow r^\star$}.

\end{proof}
\ks{A key feature of the proposed formulation is that the extreme value
need not be attained at a finite time. If it is approached only
asymptotically, it is recovered as the supremum of the lower levels
reached at finite times.}

\subsection{Numerical solution of the reachability problem}\label{sec:numerics}

The optimization problem \eqref{eq:feas_for_bisection_compact} is an infinite-dimensional linear program over nonnegative Radon measures supported on compact semialgebraic sets. To obtain a numerically tractable approximation, we employ the moment-SOS hierarchy and construct a hierarchy of finite-dimensional semidefinite relaxations.

Since all measures in \eqref{eq:feas_for_bisection_compact} are supported on compact semialgebraic sets (and we make them Archimedean by adding a redundant ball constraint if necessary), they admit moment representations, and the moment–SOS hierarchy based on Putinar’s Positivstellensatz \cite{putinar_positive_1993} provides a sequence of finite-dimensional relaxations. \chat{For a given relaxation order \claude{$k\in\mathbb N$}, we truncate these moment sequences to moments of total degree at most \claude{$2k$}, denoted respectively by}
\[
\mathbf y^\nu, \quad \mathbf y^{\nu_{\mathrm T}}, \quad \mathbf y^{\mu_0}.
\]

\chat{The discounted clock Liouville equation
\eqref{eq:discounted_clock_liouville} is enforced weakly against polynomial
test functions. \mk{The order-$k$ relaxation uses all monomial test functions such that all polynomials appearing in \eqref{eq:discounted_clock_liouville} are of degree at most $2k$.}


Remark~\ref{rem:polynomial_test_functions}
justifies the use of polynomial tests. The resulting identities are linear
constraints on the truncated moment sequences. The support constraints are
imposed through the corresponding moment and localizing matrices. In
particular, with
\[
h_{\mathcal B}(y):=1-\|y\|^2,
\]
the sphere equality is imposed, for each relevant moment sequence
$\mathbf y^\sigma$, by the localizing equalities
\[
\claude{M_{k-1}(h_{\mathcal B}\,\mathbf y^\sigma)=0,}
\qquad
\sigma\in\{\nu,\nu_{\mathrm T},\mu_0\}.
\]
The remaining polynomial inequalities defining $\mathcal B_r$,
$\overline{\mathcal B\setminus\mathcal B_r}$, and $\mathcal Y_0$ give the
usual positive-semidefinite localizing constraints. For their explicit
construction, see \cite{lasserre_positive_2010}. The normalization
$\mass\mu_0=1$ fixes the zeroth-order moment of $\mu_0$.}

\chat{At \claude{relaxation order $k$}, this construction yields a
finite-dimensional SDP with exact optimal value \claude{$\beta_k(r)$}. One has
\[
\claude{\beta_k(r)\geq\beta(r)\geq0,}
\]
and \claude{$(\beta_k(r))_{k\in\mathbb N}$} is nonincreasing\claude{, since the
test functions and the associated moment and localizing
constraints are nested in those of order $k+1$}. Consequently,
the exact identity \claude{$\beta_k(r)=0$} implies $\beta(r)=0$ and hence proves
that the level $r$ is unreachable.}

\chat{A floating-point solver, however, returns only an approximation of
\claude{$\beta_k(r)$}. A computed zero, or a value below some numerical threshold,
does not by itself establish the exact identity \claude{$\beta_k(r)=0$}.
Algorithm~\ref{alg:peak_bisection_short} therefore produces numerical
candidate bounds unless the zero decision is accompanied by a rigorously
verified dual SOS certificate, for example by rational reconstruction or
interval verification. Reusing the output at order \claude{$k$} as the upper end
of the bisection interval at order \claude{$k+1$} is a useful numerical strategy.}

\begin{algorithm}[H]
\caption{Extreme value estimation by reachability testing}
\label{alg:peak_bisection_short}
\begin{algorithmic}[1]
\Require Dynamics $\dot y=f_p(y)$, initial set $\mathcal Y_0$, $\alpha>0$, initial relaxation order \claude{$k_0$}, initial interval $[\underline r,\overline r]$, $\varepsilon>0$, $\delta \geq 0$
\Ensure \chat{Bound} $r_{\max}$ on the extreme value
\State \claude{$k \gets k_0$}
\State $r_{\text{max}}\gets \overline r$
\While{the relaxation of order \claude{$k$} is computationally tractable}
\State $r_{\text{min}}\gets \underline r$
    \While{$r_{\max}-r_{\min}>\varepsilon$}
        \State $r \gets (r_{\min}+r_{\max})/2$
        \State \parbox[t]{0.8\linewidth}{\claude{$\beta_k(r)$} $\gets$ solution of the order-\claude{$k$} relaxation of \eqref{eq:feas_for_bisection_compact} with parameter $r$, dynamics $f_p$, initial set $\mathcal Y_0$ and discount factor $\alpha$}
        \If {\claude{$\beta_k(r)$} $> \delta$}
            \State $r_{\min}\gets r$
        \Else
            \State $r_{\max}\gets r$
        \EndIf
    \EndWhile
    \State \claude{$k \gets k+1$}
\EndWhile

\State \Return $r_{\max}$
\end{algorithmic}
\end{algorithm}

\begin{remark}[Numerical detection of reachability]
\chat{The theoretical implication
\[
\claude{\beta_k(r)=0}\quad\Longrightarrow\quad\beta(r)=0
\]
concerns exact optimal values. In floating-point computation, both a
comparison with $\delta>0$ and a literal comparison with $\delta=0$ remain
numerical tests and need not prove exact vanishing. Accordingly, unverified
solver outputs are reported below as numerical estimates. The term
``certified upper bound'' is reserved for a level supported by a rigorously
verified dual certificate. Small negative solver outputs are treated only as
numerical artifacts, consistently with the exact nonnegativity of the
primal objective.}
\end{remark}

\chat{All computations reported in this paper use a positive threshold
$\delta$ and are therefore described as numerical estimates rather than
certified upper bounds.}

Finally, we remark that once a finite bound is obtained in the compactified variables, the subsequent bisection procedure can be carried out directly in the original state coordinates, that is, the system $\dot x = f(x), x(0)\in \mathcal X_0$. Indeed, suppose that for a given relaxation order \claude{$k$} the moment relaxation \chat{establishes, exactly or through a verified certificate,} that all admissible trajectories remain in a compact subset of the compactified state space $\mathcal B_r$ for some $r < 1$. Based on the expression \eqref{eq:peak_quantity_equivalence}, we have the relation $\|x\|^2 \le \frac{r}{1-r}$, which immediately yields a bound on the norm of the original state variable, so that the reachable set of the original system is contained in a compact subset of $\mathbb R^n$. In this situation, the extreme value estimation problem can be reformulated directly in the original coordinates without compactification, using the discounted clock Liouville's equation posed on a compact state space.

This observation allows one to restrict the use of compactification to a preliminary step whose sole purpose is to obtain a non-sharp bound on the true extreme value. Once a finite bound is established, the bisection-based reachability test can be performed in the original coordinates, avoiding the highly nonlinear change of variables (and the additional variable) induced by compactification. This typically leads to better numerical conditioning and improved performance of the moment-SOS hierarchy.

\subsection{Getting an extreme value of a general polynomial $p(x)$}
\label{sec:peak_general_polynomial}

So far, we have focused on recovering the maximum value of the squared norm $\|x\|^2$. However, the same reachability-based methodology applies to the extreme value estimation of a general polynomial observable
\[
p:\mathbb R^n \to \mathbb R,
\qquad p \in \mathbb R[x].
\]
Once a finite state bound has been established, extending the method to arbitrary polynomial observables requires only replacing the level sets of the norm by level sets of $p(x)$. Our goal is to compute
\[
p^\star := \chat{\sup\left\{p(x(t)):\ x(\cdot)\text{ is admissible},\ t\geq0\right\}}
\]
over all admissible trajectories of the original system.

Suppose that by the compactified reachability procedure described above we have obtained a bound
\[
\sup_{t \ge 0} \|x(t)\| \le R
\]
for all admissible trajectories, for some finite $R>0$. Then the reachable set of the original system is contained in the compact ball $\mathcal X_R = \Set{x \in \R^n}{\|x\|\leq R} \subset \mathbb R^n$, and $p(x)$ admits finite bounds on this set that can be computed explicitly.

Let
\[
p(x)=\sum_{\alpha\in\mathbb N^n} c_\alpha x^\alpha
\]
be the monomial expansion of $p$, and denote by $d_p$ the degree of $p(x)$. For all $x\in \mathcal X_R$, we have the estimate
\[
|x^\alpha| \le \|x\|^{|\alpha|} \le R^{|\alpha|},
\]
which yields the explicit uniform bound
\[
|p(x)| \le \sum_{\alpha} |c_\alpha| R^{|\alpha|}.
\]
Consequently, a valid initialization interval for the bisection is given by
\[
\underline r := -\sum_{\alpha} |c_\alpha| R^{|\alpha|},
\qquad
\overline r := \sum_{\alpha} |c_\alpha| R^{|\alpha|},
\]
which satisfies
\[
\underline r \le p(x) \le \overline r
\qquad \text{for all } x \in \mathcal X_R.
\]
These bounds are generally non-sharp, but they are explicit, and sufficient to initialize the bisection procedure.

As in Section~\ref{sec:bisection}, we avoid maximizing the discounted terminal cost directly and instead formulate extreme value estimation as a sequence of reachability tests. For a candidate level $r \in \mathbb R$, we ask whether there exists an admissible trajectory such that
\[
p(x(t)) \ge r
\]
at some time. This is implemented by imposing the support constraints
\[
\mathrm{spt}\,\nu_{\mathrm T} \subset \{x \in \mathcal X_R \mid p(x) \ge r\}, \qquad \mathrm{spt} \nu \subset \{x \in \mathcal X_R \mid p(x) \le r\}.
\]

As before, if the corresponding discounted clock Liouville's problem admits a feasible solution with $\mass\,\nu_{\mathrm T}>0$, then the level $r$ is reachable by at least one admissible trajectory; if the optimal value is zero, then $r$ is an upper bound on the extreme value of $p(x)$. This leads naturally to a bisection procedure on the inequality $p(x)\le r$.

The bisection then proceeds exactly as in Section~\ref{sec:bisection}: at each step, we test reachability of the superlevel set $\{x \in \mathcal X_R\mid p(x)\ge r\}$ for the midpoint $r=(r_{\max}+r_{\min})/2$ using a moment-SOS relaxation of the corresponding measure LP. We initialize the bisection with $r_{\min} = \underline r$, $r_{\max} = \overline r$.

In summary, once a finite bound on the state has been established, the proposed reachability-based framework enables us to recover extreme values of arbitrary polynomial observables $p(x)$ using only explicit norm-based bounds and without introducing any additional optimization layer.

\section{Examples}\label{sec:examples}
In this section, we illustrate the developed methodology in a few examples. First, we show that we are able to find an upper bound for the extreme value occurring asymptotically as time goes to infinity. This will be shown on the Van der Pol oscillator \cite{khalil}. Second, we show that when a trajectory from the feasible set blows up, the upper bound never drops below 1 (which corresponds to infinity) for all relaxations we are able to solve. In the same example, we also show the computation of upper bounds for finite extreme value attained at finite time. We use \texttt{GloptiPoly} \cite{henrion_gloptipoly_2009} to model the measure relaxations. We then use \texttt{Mosek} \cite{mosek} solver to solve the relaxations numerically.

\begin{example}[Van der Pol oscillator]\label{ex:vdp}
    
Van der Pol oscillator is described by the following equations
\begin{align}
    \dot x_1(t) &= x_2\\
    \dot x_2(t) &= -x_1 + (1-x_1^2)x_2.
\end{align}
We are interested in finding the extreme value of the Euclidean norm of the state for the initial condition $x(0) = (1/2, 1/2)$. The compactified dynamics after the time reparametrization corresponding to system \eqref{z_der_poly}-\eqref{w_der_poly} is 
\begin{align*}
\dot z_1(\tau) &= z_2(-z_2w^2z_1 + w^2 + z_2z_1^3)\\
\dot z_2(\tau) &= - w^2z_1 - w^2z_2^3 + w^2z_2 + z_1^2z_2^3 - z_1^2z_2\\
\dot w(\tau) &= -wz_2^2(w^2 - z_1^2).
\end{align*}
This vector field corresponds to $f_{\text p}$ from Algorithm \ref{alg:peak_bisection_short}. We choose bisection tolerance $\varepsilon = 10^{-3}$ and zero detection tolerance $\delta = 5\cdot 10^{-7}$. The initial set in the compactified variables described in Section \ref{sec_compactification} is \[\mathcal Y_0 = \Set{(z_1, z_2, w) \in \mathcal B}{z_1 - w/2 = 0, z_2 - w/2 = 0, z_1^2 + z_2^2 \leq w^2/2}.\] The last inequality comes from the homogenization of the ball constraint $x_1^2 + x_2^2 \leq 1/2$ in Assumption \ref{ass:ball_constraint}.

In accordance with Section \ref{sec:bisection}, we select the initial interval for the bisection as $[\underline r, \overline r] = [1/3, 1]$, because we can take $\underline r = \|z_0\|^2 = \frac{\|x_0\|^2}{\|x_0\|^2 + 1}$ for any point $y = (z, w) \in \mathcal Y_0$ and $\|x_0\| = \frac{\sqrt 2}{2}$. We take discount factor $\alpha = 1$, $\varepsilon = 10^{-3}$, $\delta = 5\cdot 10^{-7}$. Now we can use Algorithm \ref{alg:peak_bisection_short}. The computed upper bounds $r_{\text{max}}$ from the outer loop of the Algorithm \ref{alg:peak_bisection_short} are shown in Table \ref{tab:vdp_comp}. \mk{The value $k$ is the relaxation order, in the sense of Section~\ref{sec:numerics}}. The number is $r_k$ shows the computed upper bound for $r^{\star}$ in the compactified variables, $p_k$ is the upper bound for the Euclidean norm in the original coordinates computed by relation \eqref{eq:peak_quantity_equivalence} from $r_k$ and $p_k'$ denotes the computation of the upper bound in the original coordinates for relaxations after an upper bound was computed in degree 8. We can see that the upper bounds obtained by computations in original coordinates are sharper. This is caused by the fact that compactification is a highly nonlinear change of variables, as discussed in Section \ref{sec:numerics}. Note that we can refine the a priori bound on the state norm after we obtain a new (tighter) bound each relaxation. 


\begin{table}[h]
\centering
\caption{Numerical upper bounds for the Van der Pol oscillator for different relaxation orders $k$. The number $r_k$ is the upper bound for $r^{\star}$ in the compactified variables, $p_k$ is the upper bound in the original coordinates from relation \eqref{eq:peak_quantity_equivalence}, and $p_k'$ is the computation of the upper bound in the original coordinates for relaxations (with bisection tolerance $10^{-3}$) after an upper bound was computed with $k = 8$.}
\begin{tabular}{c|c|c|c} \label{tab:vdp_comp}
     $k$ & $r_k$ & $p_k$ & $p_k'$\\
     \hline
     2 & 1 & $\infty$ & -\\
     3 & 1 & $\infty$ & -\\
     4 & 0.999 & 39.179 & 4.877\\
     5 & 0.980 & 6.971 & 3.160\\
     6 & 0.967 & 5.377 & 2.926\\
\end{tabular}
\end{table}
By simulating the system with the adaptive Runge-Kutta method \texttt{ode45} in \texttt{MATLAB}, we obtained the (lower bound for the) extreme value 2.8326. Since the initial point $(1/2, 1/2)$ lies inside the limit cycle, the extreme value coincides with the maximal norm of the limit cycle (see e.g. \cite{khalil}). Hence, the supremum is achieved asymptotically as $t \to \infty$ and is not attained at any finite time, which highlights the necessity of the proposed infinite-horizon formulation.
\end{example}

\begin{example}[Finite time blow-up]
In this example, we consider the decoupled system of ODEs
\begin{align}
    \dot x_1(t) &= -x_1 \\
    \dot x_2(t) &= x_2^2.
\end{align}
We are interested in estimating the maximal Euclidean norm of the solution of this system with initial condition $x_0 \in \mathcal X_0 = \{0\} \times [0, 1]$. The compactified dynamics after time reparametrization (corresponding to \eqref{z_der_poly} - \eqref{w_der_poly} and $f_{\text p}$ of Algorithm \ref{alg:peak_bisection_short}) is 
\begin{align*}
    \dot z_1(\tau) &= -z_1(-wz_1^2 + z_2^3 + w)\\
    \dot z_2(\tau) &= z_2(wz_1^2 - z_2^3 + z_2)\\
    \dot w(\tau) &= w(wz_1^2 - z_2^3).
\end{align*}
We know that for any point from the initial set $(0, x_{20})$ with $x_{20} > 0$, there is a finite time blow-up. Hence, we expect the proposed method to never drop the upper bound below 1 in the compactified coordinates as $r^{\star} = 1$. With $\alpha = 1$, $\varepsilon = 10^{-3}$, $\delta = 5\cdot 10^{-7}$, \[\mathcal Y_0 = \Set{(z, w)\in \mathcal B}{z_1 = 0, z_2 \geq 0, -z_2 + w \geq 0, z_1^2 + z_2^2 - w^2 \leq 0},\] and $[\underline r, \overline r] = [1/2, 1]$, where $\underline r$ was chosen using the point $(0, 1) \in \mathcal X_0$ by formula \eqref{eq:peak_quantity_equivalence}, the Algorithm \ref{alg:peak_bisection_short} shows precisely this expected behaviour. The last inequality in the definition of the set $\mathcal Y_0$ again comes from Assumption \ref{ass:ball_constraint}. The upper bounds provided by the outer loop of Algorithm \ref{alg:peak_bisection_short} are written in Table \ref{tab:blow_up}, where $k$ \mk{denotes the relaxation order as defined in  Section~\ref{sec:numerics}}, $r_k$ is the computed upper bound for $r^{\star}$ in the compactified variables, and $p_k$ is the upper bound for the Euclidean norm in the original coordinates (by relation \eqref{eq:peak_quantity_equivalence} from $r_k$). We emphasize that the upper bound never dropping for finite $k$ is not a certificate of blow-up. 
\begin{table}[h]
\centering
\caption{Numerical upper bounds for finite time blow-up problem. Here $k$ is the relaxation order and $r_k$ is the upper bound for $r^{\star}$ in the compactified variables, and $p_k$ is the upper bound in the original coordinates from relation \eqref{eq:peak_quantity_equivalence}.}
\begin{tabular}{c|c|c} \label{tab:blow_up}
     $k$ & $r_k$ & $p_k$ \\
     \hline
     2 & 1 & $\infty$ \\
     3 & 1 & $\infty$ \\
     4 & 1& $\infty$ \\
     5 & 1 & ${\infty}$ \\
     6 & 1 & $\infty$ \\
\end{tabular}
\end{table}

In order to show the versatility of the method, we estimate the extreme value of the same system but starting with initial condition $x_0 \in \mathcal X_0 = [0, 1]\times \{0\}$. Thus, the state $x_2$ that blew up in the previous experiment is set to 0 and the state $x_1$ exponentially decays to 0 from any initial condition. Hence, the extreme value of this system is attained at time $0$ and it is $\max_{x_{10} \in [0, 1]}\|(x_{10}, 0)\| = 1$. According to formula \eqref{eq:peak_quantity_equivalence}, this corresponds to extreme value $r^{\star} = 1/2$ in the compactified coordinates. We use 
$\alpha = 1$, $\varepsilon = 10^{-3}$, $\delta = 5\cdot 10^{-7}$, \[\mathcal Y_0 = \Set{(z, w)\in \mathcal B}{z_2 = 0, z_1 \geq 0, -z_1 + w \geq 0, z_1^2 + z_2^2 - w^2 \leq 0},\] and $[\underline r, \overline r] = [1/3, 1]$ for Algorithm \ref{alg:peak_bisection_short}. The last inequality in the definition of the set $\mathcal Y_0$ comes from the ball constraint in Assumption \ref{ass:ball_constraint}. The values of the upper bounds $r_{\text{max}}$ from the outer loop of Algorithm \ref{alg:peak_bisection_short} are shown in Table \ref{tab:blow_up_finite}. As in the previous examples, \mk{$k$ is the relaxation order and $r_k$ is the upper bound for $r^{\star}$ in the compactified variables}, $p_k$ is the upper bound for the Euclidean norm in the original coordinates from relation \eqref{eq:peak_quantity_equivalence}, and $p_k'$ is the computation of the upper bound in the original coordinates for relaxations after an upper bound was computed with $k = 8$.
\begin{table}[h]
\centering
\caption{Numerical upper bounds for finite extreme value in the blow-up problem. Here $k$ is the relaxation order and $r_k$ is the upper bound for $r^{\star}$ in the compactified variables, $p_k$ is the upper bound in the original coordinates from relation \eqref{eq:peak_quantity_equivalence}, and $p_k'$ is the computation of the upper bound in the original coordinates for relaxations after an upper bound was computed with $k = 8$.}
\begin{tabular}{c|c|c|c} \label{tab:blow_up_finite}
     $k$ & $r_k$ & $p_k$ & $p_k'$\\
     \hline
     2 & 1 & $\infty$ & -\\
     3 & 1 & $\infty$ & -\\
     4 & 0.992 & 11.269 & 3.688\\
     5 & 0.986 & 8.523 & 2.275\\
     6 & 0.979 & 6.798 &1.734\\
     7 & - & - & 1.515\\
     8 & - & - & 1.398\\
\end{tabular}
\end{table}

As in Example \ref{ex:vdp},  we can see that returning to original coordinates after a finite bound is computed with $k = 8$ yields much better results. Since there are fewer variables in the original dynamics and the degree of the vector field is lower, we are able to increase the degree of test functions $k$ in the original coordinates compared to the compactified coordinates to obtain sharper bounds.

\end{example}

\section{Conclusion and perspectives}
In this work, we proposed a reachability based formulation of the extreme value estimation problem over an infinite time horizon. This work can also be used to obtain non-sharp upper bounds for the state norm, which is valuable in its own right since many existing methods rely on the availability of such a priori bounds. For the infinite dimensional measure LP, we proved a no relaxation gap result by establishing a superposition principle adapted to the infinite horizon setting with a time distributed terminal measure. The proposed method naturally handles both noncompact state spaces and infinite time horizons within a unified measure theoretic framework. The resulting formulation is posed entirely on compact semialgebraic sets and therefore admits systematic approximation by the moment-SOS hierarchy.

The main practical challenge lies in the numerical implementation. Besides the discount factor $\alpha$, the algorithm requires numerical tolerances $\delta$ and $\varepsilon$, whose choice influences both efficiency and conservativeness of the computed bounds. In particular, choosing $\delta$ is especially delicate as setting the threshold too low can make the upper bounds unnecessarily conservative, while setting it too high could make the upper bound drop below the true extreme value. Setting it close to the solver's tolerance seems a reasonable choice.

The algorithm proposed in this work outputs a sequence of numerical upper bounds on the extreme value. However, we could not prove that this sequence converges to the extreme value. In fact, our preliminary analysis indicates that no such algorithm exists. Concretely, we conjecture that the extreme value of a dynamical system with a polynomial vector field over an infinite time horizon is not upper semicomputable in the Turing's computation model.

Possible future directions could include extension of this framework to controlled and time varying systems or propose a polynomial optimization formulation which produces a true certificate of blow up. In addition, a detailed analysis of the influence of the discount factor $\alpha$ could be done, where adaptive strategies could even be considered to choose the discount factor instead of fixing it a priori. Another possible future direction could be the analysis of the robustness of the numerical solution with respect to the parameter $\delta$. Moreover, a possible extension could be to weaken the assumptions of the superposition principle and prove it in a more general setting.

\section{Acknowledgments}
The authors would like to thank Didier Henrion, Martin Kru\v z\'{\i}k and Rodolfo R\'{\i}os-Zertuche for their valuable ideas and discussions throughout the development of this contribution. We would especially like to thank Rodolfo for his comments and long discussions regarding the superposition principle. This work was funded by the European Union/M\v{S}MT \v{C}R  under the ROBOPROX project (reg.~no.~CZ.02.01.01/00/22 008/0004590) and by the Grant Agency of the Czech Technical University in Prague (grant No. SGS25/145/OHK3/3T/13). The authors acknowledge the use of AI for assistance with brainstorming ideas, mathematical development, coding and drafting the manuscript. The final content, analysis and conclusions remain the sole responsibility of the authors.

\appendix

\section{Proof of Theorem \ref{thm:superposition}}

\begin{proof}
 Extend \(f\) from $\X$ to all of $\R^n$ in such a way that this extension (still denoted by $f$) is globally Lipschitz. Let \(\Phi_t:\mathbb R^n\to\mathbb R^n\), \(t\in\mathbb R\), denote its flow, so that \(\Phi_t(x_0)=\gamma(t)\), where \(\gamma(0)=x_0 \in \X_0\). In particular, \(\Phi_{-t}\) denotes the backward flow. Since the measures in Theorem~\ref{thm:superposition} are supported on $\X$, these measures satisfy the Liouville's equation~\eqref{eq_liouville_standard} with the extended vector field. At the end of the proof, we verify that the resulting trajectories in the support of the constructed measure $\eta$ remain in \(\mathcal X\), so the proof is independent of the chosen extension. 
 
 From now on, we suppose that $f$ is $C^1$ and globally Lipschitz. The case of $f$ being only globally Lipschitz follows by standard mollification arguments.

Throughout the proof we assume that we are given a triplet of measures $(\mu_0, \mu, \mu_{\mathrm T})$ satisfying the assumptions of Theorem~\ref{thm:superposition}, in particular satisfying Liouville's equation~\eqref{eq_liouville_standard}. We regard these measures as measures on the ambient Euclidean spaces, extended by zero outside of their supports.
 
\paragraph{Relation between $\mu_0$ and $\mu_{\mathrm T}$} In the first step, we derive a relation between the measures $\mu_0$ and $\mu_{\mathrm T}$. Write
\[
L\varphi(t,x)
:=
\frac{\partial \varphi(t,x)}{\partial t}+\nabla_x\varphi(t,x)\cdot f(x).
\]

We first identify the part of the initial mass which stops at a finite time.
For \((\tau,x)\in\Rp\times \R^n\), set
\[
B(\tau,x):=(0, \Phi_{-\tau}(x)).
\]
We claim that
\begin{equation}
B_\#\mu_{\mathrm T}\le \mu_0 .
\end{equation}
Let \(a\in C^1_c(\R^n)\), \(a\ge0\). Let \(\chi_R\in C_c^1(\Rp)\) satisfy
\[
0\le \chi_R\le1,\qquad
\chi_R(t)=1\ \text{for }0\le t\le R,\qquad
\chi_R'(t)\le0,\qquad \chi_{R} \le \chi_{R+1}
\]
and choose the family so that $\lim_{R\to\infty}\chi_R(t)=1$ for every
\(t\in\Rp\). Use the test function 
\[
\varphi_R(t,x):=\chi_R(t)a(\Phi_{-t}(x)).
\]
Since \(L(a(\Phi_{-t}(x)))=0\), we have
\[
L\varphi_R(t,x)=\chi_R'(t)a(\Phi_{-t}(x)).
\]
Inserting \(\varphi_R\) in the weak equation gives
\[
\int_{\Rp\times\R^n}\chi_R'(t)a(\Phi_{-t}(x))\,d\mu(t,x)
+
\int_{\{0\}\times\Xzero}a(x)\,d\mu_0(t, x)
-
\int_{\Rp\times\R^n}\chi_R(t)a(\Phi_{-t}(x))\,d\mu_{\mathrm T}(t,x)
=0.
\]
Because \(a\ge0\) and \(\chi_R'\le0\), it follows that
\[
\int_{\Rp\times\R^n}\chi_R(t)a(\Phi_{-t}(x))\,d\mu_{\mathrm T}(t,x)
\le
\int_{\{0\}\times\Xzero}a(x)\,d\mu_0(t, x).
\]
Letting \(R\to\infty\) and using monotone convergence yields
\begin{equation}
\int_{\Rp\times\R^n}a(\Phi_{-\tau}(x))\,d\mu_{\mathrm T}(\tau,x)
\le
\int_{\{0\}\times\Xzero}a(x)\,d\mu_0(t, x).
\end{equation}
This proves \(B_\#\mu_{\mathrm T}\le\mu_0\). 
In particular, since pushforward preserves total mass and \(B_{\#}\mu_{\mathrm T}\leq\mu_0\), it follows that \(\mu_{\mathrm T}(\mathbb R_+\times\mathcal X)\leq\mu_0(\{0\}\times\mathcal X_0)<\infty\).

\paragraph{Residual measure} Define the residual measure
\begin{equation}
\chat{\mu_\infty^0:=\mu_0-B_\#\mu_{\mathrm T}.}
\end{equation}
Thus \chat{\(\mu_\infty^0\)} is a finite measure supported on \ks{\(\{0\}\times\Xzero\)}. It
represents the part of the initial mass which never reaches the finite-time
terminal measure.

\paragraph{Candidate occupation measure} Define a positive Radon measure \(\bar\mu\) on \(\Rp\times\R^n\) by
\begin{align}\label{eq:mubar_rep}
\int_{\Rp\times\R^n} h(t,x)\,d\bar\mu(t,x)
:={}&
\int_{\Rp\times\R^n}
\left(
\int_0^\tau h(s,\Phi_{s-\tau}(x))\,ds
\right)d\mu_{\mathrm T}(\tau,x)
\\
&+
\int_{\{0\}\times\Xzero}
\left(
\int_0^\infty h(s,\Phi_s(x))\,ds
\right)d\chat{\mu_\infty^0}(\tau, x)\notag
\end{align}
for every bounded compactly supported Borel function \(h:\Rp\times\R^n\to\R\). 
We claim that \((\mu_0,\bar\mu,\mu_{\mathrm T})\) satisfies the Liouville's equation~\eqref{eq_liouville_standard}. Let \(\varphi\in C_c^1(\Rp\times\R^n)\) have compact support in the time
variable. For the finite-terminal part,
\[
\int_0^\tau
L\varphi(s,\Phi_{s-\tau}(x))\,ds
=
\varphi(\tau,x)-\varphi(0,\Phi_{-\tau}(x)).
\]
For the non-stopping part,
\[
\int_0^\infty
L\varphi(s,\Phi_s(x))\,ds
=
-\varphi(0,x),
\]
because \(\varphi\) vanishes for all sufficiently large \(s\). Therefore
\begin{align*}
\int_{\Rp\times\R^n}L\varphi\,d\bar\mu
={}&
\int_{\Rp\times\R^n}\varphi(\tau,x)\,d\mu_{\mathrm T}(\tau,x)
-
\int_{\Rp\times\R^n}\varphi(0,\Phi_{-\tau}(x))\,d\mu_{\mathrm T}(\tau,x)
\\
&-
\int_{\{0\}\times\Xzero}\varphi(0,x)\,d\chat{\mu_\infty^0}(\tau, x).
\end{align*}
Using \chat{\(B_\#\mu_{\mathrm T}+\mu_\infty^0=\mu_0\)}, we obtain
\[
\int_{\Rp\times\R^n}L\varphi\,d\bar\mu
=
\int_{\Rp\times\R^n}\varphi(t,x)\,d\mu_{\mathrm T}(t,x)
-
\int_{\{0\}\times\Xzero}\varphi(0,x)\,d\mu_0(t, x).
\]
Hence
\begin{equation}
\int_{\Rp\times\R^n}L\varphi\,d\bar\mu
+
\int_{\{0\}\times\Xzero}\varphi(0,x)\,d\mu_0(t, x)
-
\int_{\Rp\times\R^n}\varphi(t,x)\,d\mu_{\mathrm T}(t,x)
=0.
\end{equation}
Thus \(\bar\mu\) is feasible with the same \(\mu_0\) and \(\mu_{\mathrm T}\).

\paragraph{Uniqueness} We now prove that $\bar\mu = \mu$ using a uniqueness argument. Let
\[
\lambda:=\mu-\bar\mu .
\]
Then \(\lambda\) is a signed Radon measure and
\begin{equation}
\int_{\Rp\times\R^n} L\varphi\,d\lambda=0
\end{equation}
for every \(\varphi\in C_c^1(\Rp\times\R^n)\). Let \(h\in C_c^1(\Rp\times\R^n)\)  and
define
\begin{equation}
\varphi(t,x):=
-\int_t^\infty h(s,\Phi_{s-t}(x))\,ds .
\end{equation}
Then \(\varphi\) has \ks{compact support} and satisfies
\[
L\varphi(t,x)=h(t,x).
\]
Therefore
\[
\int_{\Rp\times\R^n} h\,d\lambda
=
\int_{\Rp\times\R^n} L\varphi\,d\lambda
=0.
\]
Since \(h\) was arbitrary and \(\lambda\) is a signed Radon measure, it follows
that
\begin{equation}
\mu=\bar\mu .
\end{equation}

\paragraph{Superposition measure} We now construct the measure on stopped and non-stopped curves. For
\((\tau,x)\in\Rp\times\R^n\), define the unique finite stopped curve
\begin{equation}\label{eq:curve1}
\gamma_{\tau,x}:[0,\tau]\to\R^n,
\qquad
\gamma_{\tau,x}(s):=\Phi_{s-\tau}(x).
\end{equation}
For \(x\in\Xzero\), define the unique infinite curve
\begin{equation}\label{eq:curve2}
\gamma_{\infty,x}:\Rp\to\R^n,
\qquad
\gamma_{\infty,x}(s):=\Phi_s(x).
\end{equation}
Note that the curves \eqref{eq:curve1} and \eqref{eq:curve2} are unique, because the vector field $f$ is Lipschitz.
Set
\begin{equation}
\eta
:=
\bigl((\tau,x)\mapsto(\tau,\gamma_{\tau,x})\bigr)_\#\mu_{\mathrm T}
+
\bigl((\tau, x)\mapsto(\infty,\gamma_{\infty,x})\bigr)_\#\chat{\mu_\infty^0} .
\end{equation}
Then
\[
(e_T)_\#\bigl(\eta|_{\{\tau<\infty\}}\bigr)=\mu_{\mathrm T},
\]
and
\[
(e_0)_\#\eta
=
B_\#\mu_{\mathrm T}+\chat{\mu_\infty^0}
=
\mu_0.
\]
Moreover, by the definition of \(\bar\mu\) and by \(\mu=\bar\mu\), for every
bounded Borel \(h:\Rp\times\R^n\to\R\) with compact support 
\[
\int_{\Rp\times\R^n}h(t,x)\,d\mu(t,x)
=
\int_{\Sinf}
\left(
\int_0^\tau h(t,\gamma(t))\,dt
\right)d\eta(\tau,\gamma(\cdot)),
\]
with the inner integral interpreted over \([0,\infty)\) if \(\tau=\infty\).

\ks{
\paragraph{Admissibility}
Let \((\tau,\gamma(\cdot))\in\mathcal S_\infty\) denote a stopped or non-stopped curve in the representation above. Since \(\mu\) is supported on \([0,\infty)\times\mathcal X\), the representation \eqref{eq:mubar_rep}, applied to a countable exhaustion of \([0,\infty)\times \X^c\) by bounded sets, implies that for \(\eta\)-almost every \((\tau,\gamma(\cdot))\), \(\gamma(t)\in\mathcal X\) for almost every \(t\in[0,\tau)\). If \(\gamma(t_0)\notin\mathcal X\) for some \(t_0<\tau\), then continuity of \(\gamma\) and closedness of \(\mathcal X\) give a nontrivial interval on which \(\gamma(t)\notin\mathcal X\), a contradiction. Hence \(\gamma(t)\in\mathcal X\) for all \(t<\tau\), and also at \(t=\tau\) when \(\tau<\infty\). Since the extension of \(f\) agrees with \(f\) on
\(\mathcal X\), we have \(\dot\gamma(t)=f(\gamma(t))\). Finally, \((e_0)_\#\eta=\mu_0\) implies \(\gamma(0)\in\mathcal X_0\) for
\(\eta\)-almost every \((\tau,\gamma(\cdot))\). Therefore, \(\eta\) is
concentrated on admissible trajectories.}
Finally,
\[
\eta(\Sinf)
=
\mu_{\mathrm T}(\Rp\times\X)+\chat{\mu_\infty^0}(\{0\}\times\Xzero)
=
\mu_0(\{0\}\times\Xzero).
\]
In particular, since \(\mu_0\) is a probability measure, \(\eta\) is a
probability measure as well.

\end{proof}

\begin{remark}[Finite-stopping case]
If
\[
B_\#\mu_{\mathrm T}=\mu_0,
\]
then \chat{\(\mu_\infty^0=0\)}, so \(\eta\) gives no mass to \(\tau=\infty\). In this case
the representation is purely by finite stopped trajectories. In particular,
this holds whenever
\[
\mu_{\mathrm T}(\Rp\times\X)=\mu_0(\{0\}\times\Xzero),
\]
because \(B_\#\mu_{\mathrm T}\le\mu_0\) and the two measures then have the same total
mass.
\end{remark}

\printbibliography

\end{document}